\documentclass[12pt]{amsart}
\usepackage{amssymb,latexsym,amsmath,amscd,mathrsfs,yfonts,hyperref}
\usepackage{fullpage,enumerate}
\usepackage{eucal,graphicx}
\usepackage{mathtools}
\usepackage{tikz-cd}
\input xy
\xyoption{all}

\newcommand{\CC}{\mathbb C}

\newcommand{\NN}{\mathbb N}

\def\cO{\mathcal O}
\def\cM{\mathcal M}

\def\cS{\mathcal S}

\newcommand{\adj}[4]{#1\negmedspace: #2\rightleftarrows #3:\negmedspace #4}

\def\Aut{\mathrm{Aut}}

\def\Id{\mathrm{Id}}

\def\ab{\mathrm{ab}}
\def\Perf{\mathtt{Perf}}

\def\sset{\mathtt{sSet}}
\def\dset{\mathtt{dSet}}

\def\draw{\mathtt{dr}}
\def\llp{\mathrm{llp}}
\def\rlp{\mathrm{rlp}}
\def\An{\mathtt{An}}
\def\den{\mathtt{dd}}

\def\C{\mathrm{C}}

\def\ev{\mathrm{ev}}
\def\colim{\mathrm{colim}}

\def\id{\mathrm{id}}

\def\P{\mathcal P}
\def\Pinf{\mathcal{P}_\infty}

\def\f{\mathrm{f}}

\def\op{\mathrm{op}}

\def\Fun{\mathrm{Fun}}

\def\K{\mathrm{K}}

\def\iNS{\mathtt{N}\mathcal{S_*}}
\def\uNS{\mathtt{N}\mathcal{S}}
\def\NSf{\mathtt{N}\mathcal{S}{{}^\mathtt{fin}}}

\def\iNSm{\mathtt{N}\mathcal{S^{\mathtt{SM}}}}

\def\cSf{\mathcal{S}^{\mathrm{fin}}_*}
\def\pS{\mathcal{S_*}}
\def\Sp{\mathtt{Sp}}

\def\cJ{\mathcal J}

\def\Cat{\mathtt{Cat}}

\def\oper{\mathtt{Operad}}
\def\Ho{\mathrm{Ho}}

\def\KK{\mathrm{KK}}

\def\Set{\mathtt{Set}}

\def\cC{\mathcal C}

\def\1{\bf{1}}

\def\cH{\mathcal H}

\def\path{\mathrm{Path}}

\def\Csep{\mathtt{SC^*}}

\def\iCsep{\text{$\mathtt{SC_\infty^*}$}}

\def\cL{\mathcal L}

\def\cD{\mathcal D}

\def\SCu{\mathtt{SC^*_{un}}}
\def\mix{\mathrm{mix}}

\def\dN{\mathrm{N}_\mathrm{d}}
\def\kk{\mathrm{kk}}

\def\CW{\mathtt{CW^{fin}}}
\def\cA{\mathcal{A}}

\newcommand{\map}{\rightarrow}
\newcommand{\functor}{\rightarrow}

\def\mono{\hookrightarrow}

\def\Ind{\mathrm{Ind}}

\def\edge{\mathrm{edge}}
\def\Com{\mathrm{US}}
\def\N{\mathrm{N}}

\def\one{\mathbf{1}}

\def\rB{\mathrm{B}}

\newcommand{\beq}{\begin{eqnarray}}
\newcommand{\beqn}{\begin{eqnarray*}}
\newcommand{\eeq}{\end{eqnarray}}
\newcommand{\eeqn}{\end{eqnarray*}}

\theoremstyle{definition}
\newtheorem{thm}{Theorem}[section]
\theoremstyle{definition}

\newtheorem*{Rem}{Remark}

\newtheorem{lem}[thm]{Lemma}
\newtheorem{prop}[thm]{Proposition}

\newtheorem{ex}[thm]{Example}
\newtheorem{defn}[thm]{Definition}
\newtheorem{rem}[thm]{Remark}

\begin{document}

\title{$C^*$-algebraic drawings of dendroidal sets}
\author{Snigdhayan Mahanta}
\email{snigdhayan.mahanta@mathematik.uni-regensburg.de}
\address{Fakult{\"a}t f{\"u}r Mathematik, Universit{\"a}t Regensburg, 93040 
Regensburg, Germany.}
\thanks{This research project was supported by the Deutsche 
Forschungsgemeinschaft (SFB 1085), ERC through AdG 267079, and the Alexander von 
Humboldt Foundation (Humboldt Professorship of Michael Weiss).}

\begin{abstract}
In recent years the theory of dendroidal sets has emerged as an important 
framework for higher algebra. In this article we introduce the concept 
of a {\em $C^*$-algebraic drawing} of a dendroidal set. It depicts a dendroidal 
set as an object in the category of presheaves on $C^*$-algebras. We show that 
the construction is functorial and, in fact, it is the left adjoint of a Quillen 
adjunction between combinatorial model categories. We use this construction to produce a 
bridge between the two prominent paradigms of noncommutative geometry via 
adjunctions of presentable $\infty$-categories, which is the primary motivation 
behind this article. As a consequence we obtain a single mechanism to construct 
bivariant homology theories in both paradigms. We propose a 
(conjectural) roadmap to harmonize algebraic and analytic (or topological) 
bivariant $\K$-theory. Finally, a method to analyse graph algebras in terms of 
trees is sketched. 
\end{abstract}

\maketitle

\setcounter{tocdepth}{2}
\tableofcontents

\section{Introduction}
Dendroidal sets provide a convenient model for 
$\infty$-operads (see \cite{HeuHinMoe} for a comparison with Lurie's model 
\cite{LurHigAlg} for $\infty$-operads without constants). The category of 
dendroidal sets $\dset$ was introduced by Moerdijk--Weiss \cite{MoeWei1,MoeWei2} 
so that (inter alia) it can serve as a receptacle for the nerve functor on the 
category of operads $\oper$. The following commutative diagram is explanatory:

\beqn
\xymatrix{
\Cat \ar[r]\ar[d]_\N & \oper\ar[d]^{\dN}\\
\sset \ar[r] & \dset,
}
\eeqn where the vertical arrow $\N$ (resp. $\dN$) denotes the nerve (resp. 
dendroidal nerve) functor. Cisinski--Moerdijk constructed a cofibrantly 
generated model structure on $\dset$ \cite{CisMoe}, such that the fibrant 
objects are precisely the $\infty$-operads \cite{LurHigAlg}. Over the last 
decade the theory of dendroidal sets has reached an advanced stage, subsuming 
several aspects of the theory of operads and that of simplicial sets 
\cite{CisMoe1,CisMoe2}.

For a small category $\cC$ let $\P(\cC)$ denote the category of $\Set$-valued 
presheaves on $\cC$. Let $\SCu$ denote the category of {\em nonzero} separable 
unital $C^*$-algebras equipped with unit preserving $*$-homomorphisms. The 
Gel'fand--Na{\u{\i}}mark duality implies that $\SCu^\op$ can be regarded as the 
category of nonempty compact second countable noncommutative spaces with 
continuous maps. Let $\Omega$ denote the small category of trees, so that $\dset:=\P(\Omega)$ 
is the category of {\em dendroidal sets}. In this article 
we prove the following results:

\begin{enumerate}
 \item We construct a {\em noncommutative dendrices} functor $D:\Omega\functor\SCu^\op$.
 \item We construct an {\em operadic model structure} on $\P(\SCu^\op)$, an instance of Cisinksi's model structure on presheaves.
 \item We observe that the canonical adjoint pair induced by the noncommutative dendrices 
 functor via left Kan extension $$\adj{\draw}{\dset}{\P(\SCu^\op)}{\den},$$ is a Quillen 
 pair between combinatorial model categories.
\end{enumerate} We call the image of a dendroidal set applying the left adjoint functor 
$\draw:\dset\functor\P(\SCu^\op)$ the {\em $C^*$-algebraic drawing} of the dendroidal set.

These results constitute the first steps towards a bigger objective that we briefly explain below. 
There are two prevalent perspectives on noncommutative geometry - analytic and algebraic.
The analytic approach was pioneered by Connes \cite{ConBook,ConMarBook} whereas the algebraic 
approach builds upon the works of Drinfeld, Keller, 
Kontsevich, Lurie, Manin, Tabuada, To{\"e}n, and several others 
\cite{MMM,KelDG,KonNotes,LurHigAlg,TabGarden,DGNCG}. The following 
table compares the two approaches as of now:

\begin{table}[h]
\centering
\caption{Analytic vs. Algebraic}
\label{my-label}
\begin{tabular}{|l|l|l|ll}
\cline{1-3}
          & {\bf Analytic} & {\bf Algebraic} &  &  \\ \cline{1-3}
{\bf Objects} & $C^*$-algebras & $\infty$-categories &  &   \\ \cline{1-3}
{\bf Morphisms} & $*$-homomorphisms & $\infty$-functors &  &    \\ \cline{1-3}
{\bf How to subsume traditional spaces} & $X\mapsto C(X)$ & $X\mapsto\Perf_\infty(X)$ &  &    \\ \cline{1-3}
\end{tabular}
\end{table} The space $X$ above in each case must satisfy some reasonable hypotheses. The $\infty$-category 
$\Perf_\infty(X)$ is stable and in some contexts stability is included in the definition. 
This article is primarily motivated by the author's desire to reconcile the two 
viewpoints. In view of the disparate nature of the basic ingredients of 
the two paradigms a {\em bridge} between the basic objects of the two worlds in 
the form (a zigzag of) $\infty$-categorical adjunctions subject to a reasonable 
requirement (explained below) seems to be a sensible target. While constructing the bridge 
we have resorted to $\infty$-categories that reflects the state of the art.

Let $\uNS$ denote the compactly generated $\infty$-category of 
(unpointed) noncommutative spaces, whose construction is presented in subsection 
\ref{otherAdj}. The following diagram of adjunctions between presentable 
$\infty$-categories summarizes our list of results and puts them in the 
broader context (see also Remark \ref{opendSet}):

\beq  \label{corr}
\xymatrix{
&& \N(\P(\SCu^\op)^\circ)\ar@/^1pc/[lld]^{\mathbf{R}\den} 
\ar@/^1pc/@{-->}[rrd]\\
\N(\dset^\circ)\ar@/^1pc/[rru]^{\mathbf{L}\draw} &&&& \uNS . 
\ar@/^1pc/@{-->}[llu]
}
\eeq Here $\N(\cM^\circ)$ denotes the underlying $\infty$-category of a  model category $\cM$. 
The $\infty$-categorical adjuction 
$\adj{\mathbf{L}\draw}{\N(\dset^\circ)}{\N(\P(\SCu^\op)^\circ)}{\mathbf{R}\den}$ 
is induced by the Quillen adjunction 
$\adj{\draw}{\dset}{\P(\SCu^\op)}{\den}$ between combinatorial model categories mentioned 
earlier (see also Remark \ref{infAdj}). However, the dashed pair between $\uNS$ and 
$\N(\P(\SCu^\op)^\circ)$ is merely a zigzag of adjunctions that is constructed 
at the level of $\infty$-categories. This construction actually passes through a 
{\em mixed model structure}, denoted by $\P(\SCu^\op)_\mix$, on $\P(\SCu^\op)$ that is a left
Bousfield localization of the operadic model structure 
(see Definition \ref{mixModel}). Diagram \ref{corr} is our proposed {\em bridge} between the 
two paradigms of noncommutative geometry.

\subsection{Bivariant homology theories} \label{BHVision} Given any stable presentable $\infty$-category $\cC$, a 
colimit preserving functor $$\rB_\cC:\N(\P(\SCu^\op)^\circ)\functor\cC$$ can be 
viewed as a $\cC$-valued bivariant homology theory on $\N(\P(\SCu^\op)^\circ)$. 
For a presentable $\infty$-category $\cD$, let $\Sp(\cD)$ denotes its 
stabilization. The functor $\rB_\cC$ factors as  
$\N(\P(\SCu^\op)^\circ)\functor\Sp(\N(\P(\SCu^\op)^\circ))\functor\cC$.

There must be a unified framework for bivariant homology theories 
in the two paradigms of noncommutative geometry. In order to realise this objective 
one must construct a functor $\rB_\cC$ that passes the following two acid tests:

\begin{enumerate}[(i)]
 \item the composite functor $\N(\dset^\circ)\functor 
\N(\P(\SCu^\op)^\circ)\overset{\rB_\cC}{\functor}\cC$ should lead to the 
(nonconnective version of) algebraic $\K$-theory of $\infty$-operads as in 
\cite{NikKTh} and 
 \item the composite functor $\uNS \functor 
\N(\P(\SCu^\op)^\circ)\overset{\rB_\cC}{\functor}\cC$ should recover the 
opposite of the bivariant $\K$-theory of (pointed) noncommutative spaces as in 
\cite{MyColoc} after stabilization. 
\end{enumerate} Let us provide a pictorial description of our vision:

 \beq \label{vision}
\xymatrix{
\N(\dset^\circ) \ar[rd]_{\mathrm{F_1}}\ar@/^1pc/[rrrrd]^{\text{algebraic 
$\K$-theory}}\\
& \N(\P(\SCu^\op)^\circ) \ar@{-->}[rrr]^{\rB_\cC} &&& \cC .\\
\uNS 
\ar[ru]^{\mathrm{F_2}}\ar@/_1pc/[rrrru]_{\text{$\mathtt{KK_\infty}^\op$-theory}}
}
\eeq Here the functors $\mathrm{F_1}$ and $\mathrm{F_2}$ are furnished by those 
of diagram \eqref{corr}, so that $\mathrm{F_1}=\mathbf{L}\draw$. For any 
$X\in\N(\dset^\circ)$ we require $\cC(\rB_\cC\circ\mathrm{F_1}(\one), 
\rB_\cC\circ\mathrm{F_1}(X))$ to be the (nonconnective version of) algebraic 
$\K$-theory of $X$, where $\one$ is a unit object. Moreover, for any pair 
$A,B\in\uNS$ we require the following equivalence of spectra 
$$\cC(\Sp(\rB_\cC)\circ\Sp(\mathrm{F_2})(\Sigma^\infty_+(A)),
\Sp(\rB_\cC)\circ\Sp(\mathrm{F_2})(\Sigma^\infty_+(B)))\simeq 
\mathtt{KK_\infty}^\op(\mathrm{k_+}^\op(A),\mathrm{k_+}^\op(B)),$$ where 
$\mathrm{k_+}^\op$ is the composite functor 
$\uNS\functor\iNS\overset{\mathrm{k}^\op}{\functor}\mathtt{KK_\infty}^\op$ 
\cite{MyColoc}. Varying $\rB_\cC$ one can construct new bivariant homology 
theories using the above mechanism in both paradigms. For more generalities on 
bivariant homology theories of noncommutative spaces in the setting of 
$\infty$-categories and model categories the readers may refer to 
\cite{MyNSHLoc, BarJoaMe}. One possible application of this vision is outlined in 
Remark \ref{KirchbergAlg}.

\begin{Rem}
 A knowledgeable reader might contend that {\em spectral triples} constitute the 
notion of a space in noncommutative geometry {\`a} la Connes. Let us clarify 
that by a {\em space} we really mean a {\em topological space}. A spectral 
triple $(A,H,D)$ should be regarded as a noncommutative manifold, whose 
underlying topological space is determined by the $C^*$-algebra $A$. Therefore, 
our proposed bridge \eqref{corr} exists in the realm of noncommutative topology. 
\end{Rem}

\begin{Rem}
 There is also a Quillen adjunction $\adj{i_!}{\sset}{\dset}{i^*}$ that connects 
the theory of $\infty$-categories with that of $\infty$-operads. It should be 
noted that in this case the relevant model structure on $\sset$ is the Joyal 
model structure, whose fibrant objects are $\infty$-categories. Via the Yoneda 
embedding $\SCu^\op\mono\P(\SCu^\op)$ the category $\SCu^\op$ acquires a new 
class of weak equivalences from the operadic model structure on $\P(\SCu^\op)$ 
as in Definition \ref{DenModel}. We call these weak equivalences on $\SCu^\op$ 
the {\em weak operadic equivalences}. The associated homotopy theory is 
different from (the opposite of) the standard homotopy theory of 
$C^*$-algebras endowed by the $C^*$-homotopy equivalences. The exact 
difference between the two homotopy theories is not clear to the author (see Remark \ref{NewClassification}).
\end{Rem}

\begin{Rem}
The technology developed in this article works for 
all dendroidal sets. But from the viewpoint of topology it is preferable to 
restrict one's attention to {\em open dendroidal sets}, which model 
$\infty$-operads without constants (see Remark \ref{opendSet}).
\end{Rem}

\smallskip
\noindent
{\bf Notations and conventions:} Unless otherwise stated, a graph means a finite 
directed graph and a presheaf is considered to be $\Set$-valued. For the sake of 
definiteness we adopt the quasicategorical model for $\infty$-categories. An 
operad always means a coloured operad. We are mostly going to deal with the 
category of {\em nonzero} unital separable $C^*$-algebras $\SCu$ with unit 
preserving $*$-homomorphisms (except for subsection \ref{otherAdj}). Including 
the zero $C^*$-algebra from the viewpoint of trees and operads does not seem 
appropriate. 

\smallskip

\subsubsection*{Acknowledgements:} The author would like to thank U. Bunke, G. Raptis, and 
F. Trova for helpful conversations. The author is also extremely grateful to S. 
Henry and I. Moerdijk for their constructive feedback. This project was 
initiated and partially carried out by the author while visiting Max Planck 
Institute f{\"u}r Mathematik and Hausdorff Research Institute for Mathematics, 
Bonn. It is also influenced by our imagination in \cite{NCC} that was written 
under the auspices of a fellowship from Institut des Hautes {\'E}tudes 
Scientifiques, Paris in 2009. The author would also like to express 
sincere gratitude towards N. Ramachandran for rekindling the interest in this 
project. Finally, the author is indebted to the anonymous referees for their 
meticulous reports that improved and streamlined the final exposition significantly.

 \section{Dendroidal Sets} \label{dset} 
 We are going to assume familiarity with the theory of (coloured) operads and 
simplicial sets. For the uninitiated we recommend the following good sources of knowledge 
\cite{MayItLoop,BoaVog,MarShnSta,LodVal,GoeJar,BerMoe} - a list that is obviously non-exhaustive.
Since the article is written for topologists as well as operator algebraists, we review 
the theory of dendroidal sets from \cite{IWeissThesis,MoeWei1,MoeWei2,CisMoe} that is a 
simultaneous generalization of both - operads and simplicial sets. The 
exposition is quite brief and necessarily not entirely self-contained.
 
Trees have played an important role in the theory of operads ever since its 
inception. We provide an informal and very concise introduction to trees. We 
follow the nomenclature and presentation in \cite{MoeWei1,MoeLect}. A tree is a 
finite directed graph, whose underlying undirected graph is connected and 
acyclic. The vertices will be marked by $\bullet$ as shown below: 
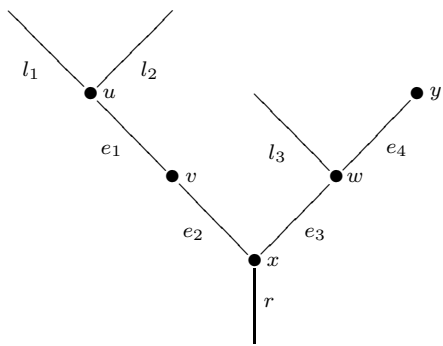
\beq
\xymatrix{*{\,}\ar@{-}[dr]_{l_1} &  & *{\,}\ar@{-}[dl]^{l_2}\\
 & *{\bullet}\ar@{-}[dr]_{e_{1}}\ar@{}|{\,\,\,\,\,\,\,\,\,\, u} &  & 
*{\,}\ar@{-}[dr]_{l_3} &  & 
*{\bullet}\ar@{-}[dl]^{e_4}\ar@{}|{\,\,\,\,\,\,\,\,\,\, y}\\
 &  & *{\bullet}\ar@{-}[dr]_{e_{2}}\ar@{}|{\,\,\,\,\,\,\,\,\,\, v} &  & 
*{\bullet}\ar@{-}[dl]^{e_{3}}\ar@{}|{\,\,\,\,\,\,\,\,\,\, w}\\
 &  &  & *{\bullet}\ar@{-}[d]^r\ar@{}|{\,\,\,\,\,\,\,\,\,\, x}\\
 &  &  & *{\,}}
\eeq An edge that is connected to two vertices is called an {\em inner edge}; 
the rest are called {\em outer edges}. Amongst the outer edges, i.e., those that 
are attached to only one vertex, there is a distinguished one called the {\em 
root}; the other outer edges are called {\em leaves}. A {\em non-planar rooted 
tree} is a non-empty tree with both inner and outer edges with the choice of one 
distinguished outer edge as the root. Henceforth, unless otherwise stated, by a 
tree we shall mean a non-planar rooted tree. Such a tree will be drawn with the 
root at the bottom and all arrows directed from top to bottom (with arrowheads 
deleted) as shown above. For instance, in the above tree there are three leaves 
$l_1,l_2,l_3$, four inner edges $e_1,e_2,e_3, e_4$, and the root is $r$. Note 
that the number of inner edges as well as leaves in a tree could be zero. The 
simplest possible tree is \beqn\xymatrix{{}\ar@{-}[d]\\ {}\ar@{}|{\,\,\,\,\,\, 
,}}\eeqn which is called the {\em unit tree}.

The category of simplicial sets, denoted by $\sset$, is the category of 
$\Set$-valued presheaves on the category of simplices $\Delta$, i.e., 
$\Fun(\Delta^\op,\Set)$. The notion of a morphism between 
trees is described in subsection \ref{facDeg}, and this allows us to define a category $\Omega$ of trees. 
Then, in analogy with simplicial sets, we define dendroidal sets to be $\dset=\Fun(\Omega^\op,\Set)$, 
the category of $\Set$-valued presheaves on $\Omega$. It will be clear from the definition of the objects 
and the morphisms of $\Omega$ that it can be viewed as a full
subcategory of the category of symmetric coloured operads. There is a fully faithful functor 
$i:\Delta\mono\Omega$ leading to an adjunction $\adj{i_!}{\sset}{\dset}{i^*}$. The 
functor $i_!$ is fully faithful and hence the category of dendroidal sets is a generalization of that of 
simplicial sets. Since $\dset=\Fun(\Omega^\op,\Set)$ it suffices to describe the 
category $\Omega$. The objects of $\Omega$ are non-planar rooted trees as 
described above. Note that in a {\em planar} rooted tree the incoming edges at 
each vertex have a prescribed linear ordering, which does not exist in a 
non-planar rooted tree. Hence each such planar (resp. non-planar) rooted tree 
generates a non-symmetric (resp. symmetric) coloured operad $\Omega[T]$. The set 
of morphisms $\Omega(S,T)$ between two non-planar rooted trees $S,T$ is by 
definition the set of coloured operad maps between $\Omega[S]$ to $\Omega[T]$. 
Thus by construction $\Omega$ is the full subcategory of the category of 
symmetric coloured operads spanned by the objects of the form $\Omega[T]$. The 
colours of the operad $\Omega[T]$ correspond to the edges of $T$ and a morphism 
between such operads is completely determined by its effect on colours. Each 
vertex $v$ of a tree $T$ with outgoing edge $e$ and a labelling of the incoming 
edges $e_1,\cdots, e_n$ defines an operation $v\in\Omega[T](e_1,\cdots, e_n; 
e)$. Consider the non-planar rooted tree $T$

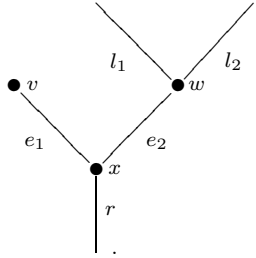
\beq
\xymatrix{
% *{\,}\ar@{-}[dr]_{l_1} &  & *{\,}\ar@{-}[dl]^{l_2}\\
  & &  & *{\,}\ar@{-}[dr]_{l_1} &  & *{\,}\ar@{-}[dl]^{l_2}\\
 &  & *{\bullet}\ar@{-}[dr]_{e_{1}}\ar@{}|{\,\,\,\,\,\,\,\,\,\, v} &  & 
*{\bullet}\ar@{-}[dl]^{e_{2}}\ar@{}|{\,\,\,\,\,\,\,\,\,\, w}\\
 &  &  & *{\bullet}\ar@{-}[d]^r\ar@{}|{\,\,\,\,\,\,\,\,\,\, x}\\
 &  &  & *{\,} \ar@{}|{\,\,\,\,\,\,\,\,\,\, .}}
\eeq The operad $\Omega[T]$ that it generates has five colours 
$l_1,l_2,e_1,e_2,$ and $r$. The generating operations are $v\in\Omega[T](;e_1)$, 
$w\in\Omega[T](l_1,l_2;e_2)$, and $x\in\Omega[T](e_1,e_2;r)$. There are also 
operations that arise from the action of the symmetric group in the non-planar 
case. For instance, if $\sigma\in\Sigma_2$, then $w\circ 
\sigma\in\Omega[T](l_2,l_1;e_2)$ is another operation. There are also the unit 
operations $1_{l_1},1_{l_2},1_{e_1},1_{e_2},$ and $1_r$ and compositions like 
$x\circ_2 w\in\Omega[T](e_1,l_1,l_2;r)$. We refrain from documenting a complete 
list of all operations and the relations they satisfy that the reader can 
herself/himself reproduce from the above diagram. Instead, we turn towards a 
more concrete (and pictorial) description of the morphisms in $\Omega$ that will 
be needed later.

\subsection{Face and degeneracy maps} \label{facDeg} We illustrate the face and 
degeneracy maps in $\Omega$ by examples that are taken directly from 
\cite{MoeWei1}, where one can find a more elaborate discussion. These maps 
provide an explicit description of all morphisms in the category $\Omega$ as we 
shall see at the end of this subsection.

\begin{enumerate}
\item If $e$ is an inner edge in $T$, then one obtains an {\em inner face} map 
$\partial_e:T/e\map T$, where $T/e$ is constructed by contracting the edge $e$ 
as shown below:

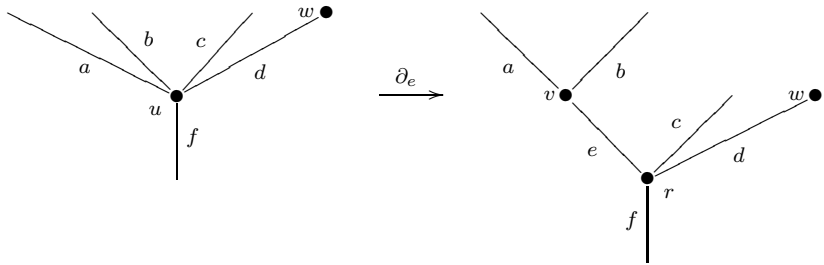
\[
\begin{array}{ccc}
\xymatrix{*{\,}\ar@{-}[rrd]_{a} & *{\,}\ar@{-}[rd]^{b} &  & 
*{\,}\ar@{-}[dl]_{c}\ar@{}[r]|{\,\,\,\,\,\,\,\,\, w} & 
*{\bullet}\ar@{-}[lld]^{d}\\
 & \,\ar@{}[r]_{\,\,\,\,\,\,\,\,\,\,\, u} & *{\bullet}\ar@{-}[d]^{f}\\
 &  & *{\,}}
 & \xymatrix{\\\ar[r]^{\partial_{e}} & *{}}
 & \xymatrix{*{\,}\ar@{-}[dr]_{a} &  & *{\,}\ar@{-}[dl]^{b}\\
\,\ar@{}[r]|{\,\,\,\,\,\,\,\,\,\,\,\,\,\, v} & *{\bullet}\ar@{-}[dr]_{e} &  & 
*{\,}\ar@{-}[dl]_{c}\ar@{}[r]|{\,\,\,\,\,\,\,\,\,\,\,\, w} & 
*{\bullet}\ar@{-}[dll]^{d}\\
 &  & *{\bullet}\ar@{-}[d]_{f} & \,\ar@{}[l]^{r\,\,\,\,\,\,\,\,\,\,\,}\\
 &  & *{\,}}
\end{array}\]

\item If a vertex $v$ in $T$ has exactly one
inner edge attached to it, one obtains the {\em outer face} map 
$\partial_v:T/v\map T$, where $T/v$ is constructed by deleting
$v$ and all the outer edges attached to it as shown below:

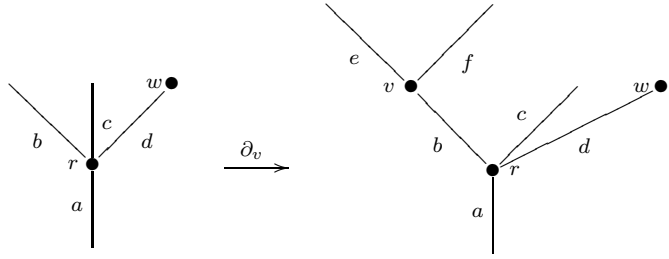
\[
\begin{array}{ccc}
\xymatrix{\\*{\,}\ar@{-}[dr]_{b} & 
*{\,}\ar@{-}[d]^{c}\ar@{}[r]|{\,\,\,\,\,\,\,\,\,\,\,\, w} & 
*{\bullet}\ar@{-}[dl]^{d}\\
\,\ar@{}[r]|{\,\,\,\,\,\,\,\,\,\,\,\, r} & *{\bullet}\ar@{-}[d]_{a}\\
 & *{\,}}
 & \xymatrix{\\\\\ar[r]^{\partial_{v}} & *{}}
 & \xymatrix{*{\,}\ar@{-}[dr]_{e} &  & *{\,}\ar@{-}[dl]^{f}\\
\,\ar@{}[r]|{\,\,\,\,\,\,\,\,\,\,\, v} & *{\bullet}\ar@{-}[dr]_{b} &  & 
*{\,}\ar@{-}[dl]_{c}\ar@{}[r]|{\,\,\,\,\,\,\,\,\,\,\,\, w} & 
*{\bullet}\ar@{-}[dll]^{d}\\
 &  & *{\bullet}\ar@{-}[d]_{a} & \,\ar@{}[l]|{r\,\,\,\,\,\,\,\,\,\,\,}\\
 &  & *{\,}}
\end{array}\]

It is also possible to remove the root and the vertex that it is attached to by 
this process as shown below:

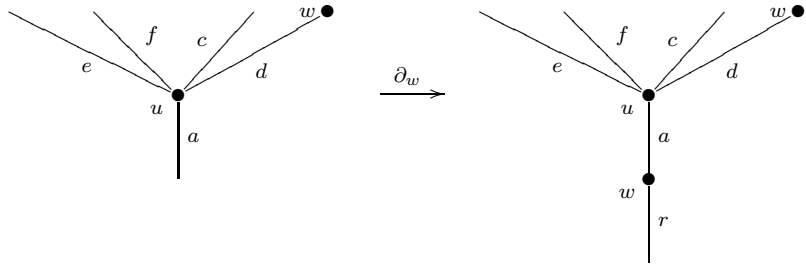
\[
\begin{array}{ccc}
\xymatrix{*{\,}\ar@{-}[rrd]_{e} & *{\,}\ar@{-}[rd]^{f} &  & 
*{\,}\ar@{-}[dl]_{c}\ar@{}[r]|{\,\,\,\,\,\,\,\,\, w} & 
*{\bullet}\ar@{-}[lld]^{d}\\
 & \,\ar@{}[r]_{\,\,\,\,\,\,\,\,\,\,\, u} & *{\bullet}\ar@{-}[d]^{a}\\
 &  & *{\,}}
 & \xymatrix{\\\ar[r]^{\partial_{w}} & *{}}
 & \xymatrix{*{\,}\ar@{-}[rrd]_{e} & *{\,}\ar@{-}[rd]^{f} &  & 
*{\,}\ar@{-}[dl]_{c}\ar@{}[r]|{\,\,\,\,\,\,\,\,\, w} & 
*{\bullet}\ar@{-}[lld]^{d}\\
 & \,\ar@{}[r]_{\,\,\,\,\,\,\,\,\,\,\, u} & *{\bullet}\ar@{-}[d]^{a}\\
 & \,\ar@{}[r]_{\,\,\,\,\,\,\,\,\,\,\, w} & *{\bullet}\ar@{-}[d]^{r}\\
 &  & *{\,}}
\end{array}\]

\item If a vertex $v\in T$ has exactly one incoming edge, there is
a tree $T\backslash v$, obtained from $T$ by deleting the vertex $v$ and merging
the two edges $e_{1}$ and $e_{2}$ on either side of $v$ into one
new edge $e$. This defines the {\em degeneracy map} $\sigma_v:T\map T\backslash 
v$ as shown below:

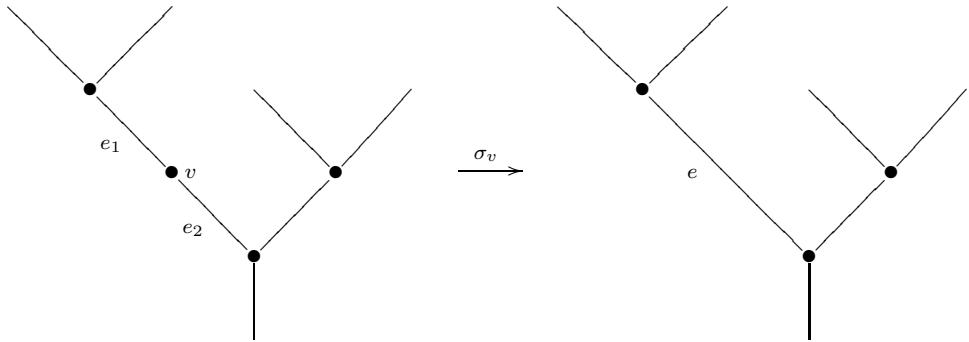
\[
\begin{array}{ccc}
\xymatrix{*{\,}\ar@{-}[dr] &  & *{\,}\ar@{-}[dl]\\
 & *{\bullet}\ar@{-}[dr]_{e_{1}} &  & *{\,}\ar@{-}[dr] &  & *{\,}\ar@{-}[dl]\\
 &  & *{\bullet}\ar@{-}[dr]_{e_{2}}\ar@{}|{\,\,\,\,\,\,\,\,\,\, v} &  & 
*{\bullet}\ar@{-}[dl]\\
 &  &  & *{\bullet}\ar@{-}[d]\\
 &  &  & *{\,}}
 & \xymatrix{\\\\\ar[r]^{\sigma_{v}} & *{}}
 & \xymatrix{*{\,}\ar@{-}[dr] &  & *{\,}\ar@{-}[dl]\\
 & *{\bullet}\ar@{-}[ddrr] &  & *{\,}\ar@{-}[dr] &  & *{\,}\ar@{-}[dl]\\
 & \,\ar@{}[r]|{\,\,\,\, e} &  &  & *{\bullet}\ar@{-}[dl]\\
 &  &  & *{\bullet}\ar@{-}[d]\\
 &  &  & *{\,}}
\end{array}\]

\end{enumerate}

The following lemma explains the importance of these maps:

\begin{lem}[Lemma 3.1 of \cite{MoeWei1}] \label{treeMap}
 Any arrow $f:S\map T$ in $\Omega$ decomposes as \beqn\xymatrix{S\ar[r]^f 
\ar[d]_\sigma & T\\ S'\ar[r]^\varphi & T',\ar[u]^\delta     }\eeqn where 
$\sigma: S\map S'$ is a composition of degeneracy maps, $\varphi: S' \map T'$ is 
an isomorphism, and $\delta: T' \map T$ is a composition of face maps.
\end{lem}

\begin{rem} \label{uniqueFac}
We have quoted the statement of Lemma \ref{treeMap} from the original source. If 
one carefully inspects its proof (cf. Lemma 2.3.2 of \cite{MoeLect}) one notices 
immediately that the factorization $f = \delta\circ\varphi\circ\sigma$ is unique.
Hence the degeneracy maps and the face maps of $\Omega$ actually constitute a factorization 
system.
\end{rem}

\subsection{Face and degeneracy identities} \label{iden} These face and 
degeneracy maps satisfy numerous identities. We illustrate them in terms of 
various commuting diagrams in $\Omega$ (with the existence of certain 
non-obvious arrows as assertions). The interested readers are referred to 
\cite{MoeWei1,MoeLect} for further details and also the discussion of a couple 
of special cases that we have left out (see Remark \ref{specCase}). 

\begin{enumerate}[(I)]
 \item If $e,f$ are distinct inner edges, then $(T/e)/f=(T/f)/e$ and the 
following diagram commutes:
 \beqn
 \xymatrix{
 (T/e)/f \ar[r]^{\partial_f}\ar[d]_{\partial_e} & T/e \ar[d]^{\partial_e} \\
 T/f \ar[r]^{\partial_f} & T.
 }\eeqn
 
 \item Assume $T$ has at least three vertices and let $\partial_v,\partial_w$ be 
distinct outer face maps. Then $(T/v)/w = (T/w)/v$ and the following diagram 
commutes:
 
 \beqn
 \xymatrix{
 (T/v)/w \ar[r]^{\partial_w}\ar[d]_{\partial_v} & T/v \ar[d]^{\partial_v} \\
 T/w \ar[r]^{\partial_w} & T.
 }\eeqn 
 
 \item If $e$ is an inner edge that is not adjacent to a vertex $v$, then 
$(T/e)/v = (T/v)/e$ and the following diagram commutes:
 
 \beqn
 \xymatrix{
 (T/v)/e \ar[r]^{\partial_e}\ar[d]_{\partial_v} & T/v \ar[d]^{\partial_v} \\
 T/e \ar[r]^{\partial_e} & T.
 }\eeqn
 
 \item Let $e$ be an inner edge that is adjacent to a vertex $v$ and let $w$ be 
the other adjacent vertex. In $T/e$ the two vertices combine to contribute a 
vertex $z$ (expressing the composition of $v$ and $w$ in some order). Then the 
outer face $\partial_z: (T/e)/z\map T/e$ exists if and only if the outer face 
$\partial_w:(T/v)/w\map T/v$ exists, and in this case $(T/e)/z = (T/v)/w$. 
Summarizing the setup the following diagram commutes:
 
 \beqn
 \xymatrix{
 (T/v)/w \ar@{=}[r]\ar[d]_{\partial_w} & (T/e)/z\ar[r]^{\partial_z} & T/e 
\ar[d]^{\partial_e} \\
 T/v \ar[rr]^{\partial_v} && T.
 }\eeqn
 
 \item If $\sigma_v,\sigma_w$ are two degeneracies of $T$, then $(T\backslash 
v)\backslash w = T\backslash w)\backslash v$ and the following diagram commutes:
 
 \beqn
 \xymatrix{
 T \ar[r]^{\sigma_v}\ar[d]_{\sigma_w} & T\backslash v \ar[d]^{\sigma_w} \\
 T\backslash w \ar[r]^{\sigma_v} & (T\backslash v)\backslash w.
 }\eeqn 
 
 \item Let $\sigma_v:T\map T\backslash v$ be a degeneracy and $\partial: T'\map 
T$ be any face map, such that $T'$ still contains $v$ and its two adjacent edges 
as a subtree. Then the following diagram commutes:
 
 \beqn
 \xymatrix{
 T \ar[r]^{\sigma_v} & T\backslash v \\
 T'\ar[u]^\partial\ar[r]^{\sigma_v} & T'\backslash v \ar[u]_\partial.
 }\eeqn 
 
 \item Let $\sigma_v: T\map T\backslash v$ be a degeneracy map and $\partial: T' 
\map T$ be a face map induced by one of the adjacent edges to v or the removal 
of $v$ (if that is possible). Then $T'= T\backslash v$ and the following diagram 
commutes:
 
 \beqn
 \xymatrix{
 T\backslash v\ar[rd]_\partial \ar[rr]^{\id_{T\backslash v}}&& T\backslash v\\
 & T \ar[ur]_{\sigma_v}.
 }
 \eeqn
\end{enumerate}

\begin{rem} \label{specCase}
 We have left out the following special cases of dendroidal identities:
 
 \begin{itemize}
  \item Outer face identities when $T$ has less than three vertices.
  
  \item Predictable identities expressing the compatibility of the face and 
degeneracy maps with isomorphisms (see, for instance, Section 2.3.1 of 
\cite{MoeLect}). 
 \end{itemize}
\end{rem}

\subsection{The model structure on $\dset$} \label{denModel}
The formalism of model categories was introduced by Quillen 
\cite{HomotopicalAlgebra} as an abstract framework for homotopy theory. For a 
modern treatment the readers may refer to \cite{Hov,Hir}. We review the model 
structure on $\dset$ constructed by Cisinski--Moerdijk \cite{CisMoe} that 
generalizes the Joyal model structure on $\sset$. 

The construction of the model structure on $\dset$ exploits the Cisinski model 
structure on any category of presheaves \cite{CisModel} (see the appendix in Section \ref{App}) 
and also a transfer principle. Typically one begins with certain desired features on the model 
structure based on intended applications. Keeping in mind 
the Joyal model structure on $\sset$ it is natural to expect that in the would 
be model structure on $\dset$ (certain) monomorphisms should be cofibrations, 
some class of objects (generalizing $\infty$-categories) should be fibrant, and 
certain morphisms (generalizing categorical equivalences) should be weak 
equivalences.

A monomorphism of dendroidal sets $X\map Y$ is {\em normal} if for any 
$T\in\Omega$, the action of $\Aut(T)$ on $Y(T)\setminus X(T)$ is free. If $e$ is 
an inner edge of a tree $T$, then one obtains an {\em inner horn inclusion} 
$\Lambda^e[T]\map\Omega[T]$, where $\Lambda^e[T]$ is obtained as the union of 
the images of all the elementary face maps apart from $\partial_e: T/e\map T$. A 
map of dendroidal sets is called an {\em inner anodyne extension} if it belongs 
to the smallest class of maps which is stable under pushouts, transfinite 
compositions and retracts, and which contains the inner horn inclusions. There 
is an adjunction $\adj{\tau_\mathrm{d}}{\dset}{\oper}{\N_\mathrm{d}}$, where 
$\tau_\mathrm{d}$ is called the {\em operadic realization} functor. The model 
structure on $\dset$ can be described as (see Theorem 2.4 of \cite{CisMoe}): 
\begin{itemize}                                                                                                                                              
\item the cofibrations are the {\em normal monomorphisms};      
\item the fibrant objects are the $\infty$-operads;
\item the fibrations between fibrant objects are the inner Kan fibrations (see 
\cite{MoeWei2} and section 2.1 of \cite{CisMoe}), whose image under 
$\tau_\mathrm{d}$ is an operadic fibration, i.e., a fibration in the canonical model structure on operads;
\item the class of weak equivalences is the smallest class $\mathtt{W}$ of maps 
in $\dset$ satisfying: \begin{enumerate}[(a)]
    \item $2$-out-of-$3$ property;
    \item inner anodyne extensions are in $\mathtt{W}$;
    \item trivial fibrations between $\infty$-operads are in 
$\mathtt{W}$.\end{enumerate}

\end{itemize}  We omit further details but explain an additional property of 
this model category that is relevant for our purposes. Let $\kappa$ be regular 
cardinal. A category $\cA$ is said to be {\em $\kappa$-accessible} if there is a 
small category $\cC$, such that $\cA\cong\Ind_\kappa(\cC)$. A {\em locally 
$\kappa$-presentable} category is a $\kappa$-accessible category that, in 
addition, possesses all small colimits. A category is {\em locally presentable} 
if it is locally $\kappa$-presentable for some regular cardinal $\kappa$. If 
$\cC$ is a small category, the category of presheaves on $\cC$ (e.g., 
$\dset=\Fun(\Omega^\op,\Set)$) is locally $\omega$-presentable (see, for 
instance, \cite{AdaRos}). Recall that a model category is said to be {\em 
combinatorial} if it is cofibrantly generated and its underlying category is 
locally presentable. It is also shown in Proposition 2.6 of \cite{CisMoe} that 
the model category $\dset$ is combinatorial. The set of generating cofibrations 
$I$ consists of the boundary inclusions of trees, i.e., 
$I=\{\partial\Omega[T]\map\Omega[T]\,|\,T\in\Omega\}$. 

\section{$C^*$-algebras associated with trees: noncommutative dendrices} 
\label{treeAlg}
The description of a tree presented in the previous section differs slightly 
from the one that one might encounter in graph theory. For instance, in the 
graph algebra literature a {\em directed graph} $G=(E^0,E^1,r,s)$ consists of 
two (countable) sets $E^0,E^1$ and functions $r,s:E^1\map E^0$. The elements of 
$E^0$ are called the {\em vertices} and the those of $E^1$ are called the {\em 
edges} of $G$. For an edge $e$, the vertex $s(e)$ is its {\em source} and the 
vertex $r(e)$ is its {\em range}. Thus in a directed graph one does not have edges 
attached only to one vertex like the leaves or the root that we considered in 
the previous section. In a graph a {\em path of length $n$} is a sequence $\mu = 
e_1e_2\cdots e_n$ of edges, such that $s(e_i)=r(e_{i+1})$ for all $i\leqslant 
i\leqslant n-1$. For such a path $\mu = e_1e_2\cdots e_n$ we denote by 
$\edge(\mu)=\{e_1,e_2,\cdots,e_n\}$ the set of all edges traversed by it. 

The $C^*$-algebra associated with a tree that we are going to describe shortly 
is to some extent inspired by the construction of noncommutative simplicial 
complexes in \cite{CunSimp}. However, we design the $C^*$-algebra from the edges 
of the tree, since from the categorical (or operadic) viewpoint the edges are 
more fundamental than the vertices. 

\begin{defn}Given a set $G$ of generators and a set $R$ of relations the {\em universal 
$C^*$-algebra}, denoted by $C^*(G,R)$, is a $C^*$-algebra equipped with a set 
map $\iota: G\map C^*(G,R)$ that satisfies the following universal property: for 
every $C^*$-algebra $A$ and a set map $\iota_A: G\map A$, such that the 
relations $R$ are fulfilled inside $A$, there is a unique $*$-homomorphism 
$\theta: C^*(G,R)\map A$ satisfying $\theta\circ\iota =\iota_A$.
\end{defn} This is a subtle concept - for instance, if $G=\{x\}$ and $R=\emptyset$, then the universal 
$C^*$-algebra $C^*(G,R)$ does not exist. In other words, free (or relation free) 
objects do not exist in the category of $C^*$-algebras. It follows from two simple facts:
\begin{enumerate}
 \item Every element in a $C^*$-algebra has a finite norm $\|\cdot\|$, i.e., a real number.
 \item Every $*$-homomorphism is norm decreasing, i.e., $\phi:A\map B$ $\implies$ $\|\phi(a)\| \leqslant \|a\|$.
\end{enumerate} If $C^*(G=\{x\},R=\emptyset)$ were to exist, then the generator $x$ would have a finite norm $\|x\|$. 
Now choose any $C^*$-algebra $A$ and an element $a\in A$ with $\|a\| > \|x\|$ that can evidently be done. 
Then it is manifestly clear that one cannot find the desired 
$*$-homomorphism $\iota: C^*(G=\{x\},R=\emptyset)\map A$ with $\iota(x) = a$ that satisfies requirement (2) above.
If the relations $R$ put a non-strict bound on the norm of 
each generator, then typically one obtains an interesting nontrivial universal 
$C^*$-algebra (although it can be trivial in certain cases). 

\begin{defn} Given any tree 
$T=(E^0,E^1)$ (viewed as a graph as described above) we define its {\em associated 
$C^*$-algebra} as the universal unital $C^*$-algebra generated by $\{q_e\,|\, e\in E^1\}$ satisfying

\begin{enumerate}
 \item $q_e\geqslant 0$ for all $e\in E^1$, 
 \item $\sum_{e\in E^1} q_{e} = 1$, and
 \item $q_{e_1}q_{e_2}\cdots q_{e_n} = 0$ unless there is a path $\mu$ with 
$\{e_1,e_2,\cdots , e_n\}\subseteq\edge(\mu)$ (inclusion of sets disregarding order).
 \end{enumerate}
 \end{defn}

\begin{rem} Let us briefly clarify the motivation behind the relations. 
\begin{itemize}
 \item The relations (1) and (2) clearly put a bound on the norm of each generator and hence the 
existence of the universal $C^*$-algebra is clear. 
\item Relation (3) encodes the compositional nature of trees. It retains those terms that lie in a path (and hence bound a simplex).
However, it also retains reorderings and repetitions of edges within the path because we want the canonical abelianization map to be surjective 
(see Remark \ref{abel} and Example \ref{abelLinear}).
\end{itemize}
\end{rem}

\begin{ex}
Note that repetitions are 
allowed amongst $e_i$'s in relation (3) above. For instance, if $T$ is 
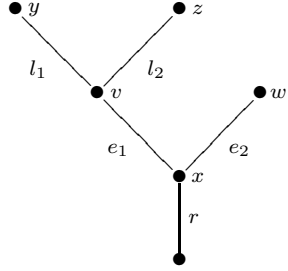
\beqn
\xymatrix{
  & *{\bullet}\ar@{-}[dr]_{l_1}\ar@{}|{\,\,\,\,\,\,\,\,\,\, y} &  & 
*{\bullet}\ar@{-}[dl]^{l_2}\ar@{}|{\,\,\,\,\,\,\,\,\,\, z}\\
 &  & *{\bullet}\ar@{-}[dr]_{e_{1}}\ar@{}|{\,\,\,\,\,\,\,\,\,\, v} &  & 
*{\bullet}\ar@{-}[dl]^{e_{2}}\ar@{}|{\,\,\,\,\,\,\,\,\,\, w}\\
 &  &  & *{\bullet}\ar@{-}[d]^r\ar@{}|{\,\,\,\,\,\,\,\,\,\, x}\\
 &  &  & *{\bullet}}
\eeqn then $q_{l_2}q_{e_1}q_{e_2} = q_{l_1}q_{l_2} = q_{e_2}q_{e_1}q_{l_2} = 0$, whereas $q_r 
q_{e_1} q_{l_1} \neq 0$ and $q_{e_1} q_{l_2} q_{e_1} \neq 0$.
\end{ex}

\noindent
Given any non-planar rooted tree $T$ we construct its associated $C^*$-algebra $D(T)$ as 
follows:                                                                         
                                                             
\begin{enumerate}[(a)]
\item insert a vertex at 
each of the top tip of the leaves (if any) and the bottom tip of the root;       
\item construct the 
universal $C^*$-algebra of the modified tree as explained above.                                                  
\end{enumerate}                                                                  
For instance given the tree

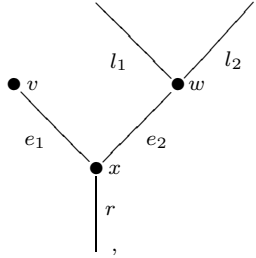
\beq
\xymatrix{
  & &  & *{\,}\ar@{-}[dr]_{l_1} &  & *{\,}\ar@{-}[dl]^{l_2}\\
 &  & *{\bullet}\ar@{-}[dr]_{e_{1}}\ar@{}|{\,\,\,\,\,\,\,\,\,\, v} &  & 
*{\bullet}\ar@{-}[dl]^{e_{2}}\ar@{}|{\,\,\,\,\,\,\,\,\,\, w}\\
 &  &  & *{\bullet}\ar@{-}[d]^r\ar@{}|{\,\,\,\,\,\,\,\,\,\, x}\\
 &  &  & *{\,}\ar@{}|{\,\,\,\,\,\,\,\,\,\, ,}}
\eeq according to procedure (a) we modify the tree as                            
 
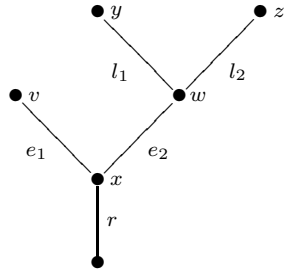
\beq
\xymatrix{
  & &  & *{\bullet}\ar@{-}[dr]_{l_1}\ar@{}|{\,\,\,\,\,\,\,\,\,\, y} &  & 
*{\bullet}\ar@{-}[dl]^{l_2}\ar@{}|{\,\,\,\,\,\,\,\,\,\, z}\\
 &  & *{\bullet}\ar@{-}[dr]_{e_{1}}\ar@{}|{\,\,\,\,\,\,\,\,\,\, v} &  & 
*{\bullet}\ar@{-}[dl]^{e_{2}}\ar@{}|{\,\,\,\,\,\,\,\,\,\, w}\\
 &  &  & *{\bullet}\ar@{-}[d]^r\ar@{}|{\,\,\,\,\,\,\,\,\,\, x}\\
 &  &  & *{\bullet}}
\eeq and then construct its universal $C^*$-algebra.

\begin{rem} \label{abel}
In the above construction we can add the relation that the generators commute, 
i.e., $q_e q_f = q_f q_e$ for all $e,f \in E^1$ to obtain a commutative 
$C^*$-algebra $D^\ab (T)$.
\end{rem}

\begin{defn}
 The $C^*$-algebra $D(T)$ associated with a non-planar rooted tree $T$ is called 
a {\em noncommutative dendrex}. Note that if $X\in\dset$ and $T\in\Omega$, then 
$X(T)$ is viewed as the set of $T$-shaped dendrices in $X$.
\end{defn}

\begin{ex} \label{abelLinear}
An object $[n]\in\Delta$ can be viewed as a linear tree $L_n$ as 
$$\leftarrow\bullet_1\leftarrow\cdots \leftarrow\bullet_n\leftarrow$$ (drawn 
horizontally instead of vertically with arrowheads inserted to indicated the 
direction). This association $[n]\mapsto L_n$ defines a fully faithful functor 
$\Delta\mono\Omega$ that produces the adjunction $\sset\rightleftarrows\dset$. 
After modification $L_n$ produces the following tree 
$$\bullet_0\leftarrow\bullet_1\leftarrow\cdots \leftarrow\bullet_{n+1},$$ whose 
associated $C^*$-algebra is the universal unital $C^*$-algebra generated by 
$n+1$ positive generators $\{q_1,\cdots, q_{n+1}\}$, such that $\sum_{i=1}^n q_i 
= 1$. Its associated commutative $C^*$-algebra (see Remark \ref{abel}) is 
isomorphic to $\C(\Delta^n)$, where $\Delta^n$ is the standard $n$-simplex (see 
Proposition 2.1 of \cite{CunSimp}). Our choice for the noncommutative dendrex 
construction was guided by this consideration. Observe that $D(L_0)=\CC$, since 
$[0]$ corresponds to the unit tree \beqn\xymatrix{{}\ar@{-}[d]\\ 
{}\ar@{}|{\,\,\,\,\,\,\,\,\,\, ,}}\eeqn whose modified tree is simply 
\beqn\xymatrix{*{\bullet}\ar@{-}[d]\\ *{\bullet}}\eeqn with only one edge. This 
phenomenon reflects the fact that the edges of a tree correspond to the colours 
of its associated operad.
\end{ex}

\subsection{Functoriality} \label{func}  
The aim of this subsection is to establish the (contravariant) functoriality of 
the above construction $T\mapsto D(T)$ with respect to morphisms of $\Omega$. To 
this end we begin by defining the $*$-homomorphisms that the faces and 
degeneracies induce. If $\sigma_v: T\map T\backslash v$ is a degeneracy map (see 
subsection \ref{treeMap}) like 

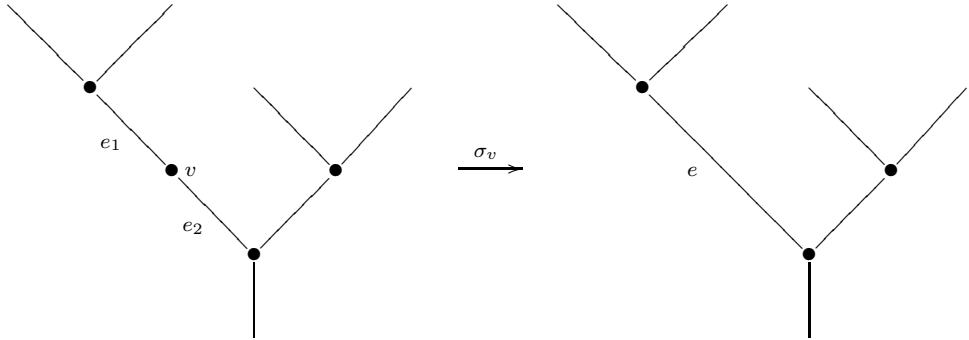
\[
\begin{array}{ccc}
\xymatrix{*{\,}\ar@{-}[dr] &  & *{\,}\ar@{-}[dl]\\
 & *{\bullet}\ar@{-}[dr]_{e_{1}} &  & *{\,}\ar@{-}[dr] &  & *{\,}\ar@{-}[dl]\\
 &  & *{\bullet}\ar@{-}[dr]_{e_{2}}\ar@{}|{\,\,\,\,\,\,\,\,\,\, v} &  & 
*{\bullet}\ar@{-}[dl]\\
 &  &  & *{\bullet}\ar@{-}[d]\\
 &  &  & *{\,}}
 & \xymatrix{\\\\\ar[r]^{\sigma_{v}} & *{}}
 & \xymatrix{*{\,}\ar@{-}[dr] &  & *{\,}\ar@{-}[dl]\\
 & *{\bullet}\ar@{-}[ddrr] &  & *{\,}\ar@{-}[dr] &  & *{\,}\ar@{-}[dl]\\
 & \,\ar@{}[r]|{\,\,\,\, e} &  &  & *{\bullet}\ar@{-}[dl]\\
 &  &  & *{\bullet}\ar@{-}[d]\\
 &  &  & *{\,}}
\end{array}\]
then define $\sigma_v^*: D(T\backslash v)\map D(T)$ as $$q_f\mapsto 
\begin{cases} q_f \text{\; if $f\neq e$,} \\ q_{e_1} + q_{e_2} \text{\; 
otherwise.}                         
\end{cases}.$$

\begin{rem}
The notation employed in the definition of $\sigma_v^*$ is potentially 
ambiguous. In the domain $q_f$ is a generator of $D(T\backslash v)$ and in the 
codomain it is a generator of $D(T)$. One should ideally differentiate them by 
writing $q_f^{T\backslash v}$ and $q_f^T$ (or something similar) to indicate the 
dependence on the tree. For notational simplicity we avoid doing this. 
\end{rem}

\begin{lem}
 The map $\sigma_v^*:D(T\backslash v)\map D(T)$ is a $*$-homomorphism.
\end{lem}

\begin{proof}
 We need to verify that the set $\{\sigma_v^*(q_f)\,|\, \text{$f$ an edge in 
$T\backslash v$}\}$ satisfies the relations (1), (2), and (3) in $D(T)$ that 
define the universal $C^*$-algebra $D(T\backslash v)$. 
 
 For (1) note that $q_{e_1}$ and $q_{e_2}$ are both positive in $D(T)$ whence so 
is $q_{e_1} + q_{e_2}$. Clearly each $q_f$ is also positive in $D(T)$. Let 
$E^1(T)$ be the set of edges in $T$. We verify (2) by computing $$\sum_{f\in 
E^1(T\backslash v)}\sigma_v^* (q_f) = \sum_{f\neq e} q_f + (q_{e_1} + q_{e_2}) 
=\sum_{f\in E^1(T)} q_f = 1.$$ For (3) one can check by inspection that if 
$f_1,f_2$ are two edges in $T\backslash v$ that do not lie in a path, then they 
cannot lie in a path in $T$.
\end{proof}

Note that every face map can be viewed as an injective map on edges (or colours 
of the associated operad). Thus if $\partial_e: T/e\map T$ is an inner face map 
then define a $*$-homomorphism $\partial_e^*: D(T)\map D(T/e)$ as $$q_f\mapsto 
\begin{cases} q_f \text{\; if $f\neq e$,} \\ 0 \text{\; otherwise.}                         
\end{cases}.$$ Similarly, if $\partial_v: T/v\map T$ is an outer 
face map then define $\partial_v^*: D(T)\map D(T/v)$ as $$q_f\mapsto 
\begin{cases} q_f \text{\; if $f$ has not been removed,} \\ 0 \text{\; 
otherwise.}                                                                      
\end{cases}.$$ 

\begin{lem}
The maps $$\partial_e^*: D(T)\map D(T/e) \;  \text{ and } \;  \partial_v^*: D(T)\map D(T/v)$$ are $*$-homomorphisms.
\end{lem} 

\begin{proof}
One needs to again verify that the set $\{\partial_e^*(q_f)\,|\, \text{$f$ an 
edge in $T$}\}$ satisfies the relations (1), (2), and (3) in $D(T/e)$ that 
define the universal $C^*$-algebra $D(T)$.  Relations (1) and (2) are clearly satisfied; for relation (3) 
one needs to observe that if two edges $e,f$ in $T$ do not lie in a path, then 
this property continues to hold in $T/e$ or $T/v$. A similar argument is applicable to 
$\partial_v^*$.
\end{proof}

\begin{rem}
Observe that if $\theta: S\map T$ is an isomorphism in $\Omega$ then $\theta^*: 
D(T)\map D(S)$ acts on the generators as $q_e\mapsto q_{\theta^{-1}(e)}$. One 
can readily verify that $\theta^*$ is a unital $*$-homomorphism. 
\end{rem}
Let $\SCu$ denote the category of separable unital $C^*$-algebras with unit preserving 
$*$-homomorphisms. Extending the Gel'fand--Na{\u{\i}}mark duality $\SCu^\op$ is 
regarded as the category of compact Hausdorff noncommutative spaces with 
continuous maps.

\begin{prop} \label{functor}
 The association of a noncommutative dendrex with a tree $T\mapsto D(T)$ defines 
a functor $D:\Omega\functor\SCu^\op$.
\end{prop}

\begin{proof}
In view of Lemma \ref{treeMap} it suffices to show that the $*$-homomorphisms 
$\partial_e^*,\partial_v^*,\sigma_v^*$ and $\theta^*$ satisfy the face and 
degeneracy identities (see subsection \ref{iden}). Note that thanks to the 
universal property of universal $C^*$-algebras we simply need to verify that 
various combinations of these $*$-homomorphisms governed by the identities agree 
on generators.

It is easy to verify that identities (I), (II), (III), and (V) are satisfied. 
The point is to observe that the order in which a certain number of generators 
are sent to $0$ or sums of two other generators does not affect the final 
outcome.

For (IV) let us suppose that the tree around $e$ looks like below

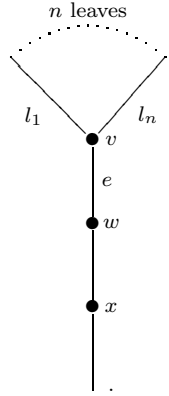
\beqn
\xymatrix{*{}\ar@{..}@/^1pc/[rr]^{\text{$n$ leaves}}\ar@{-}[dr]_{l_{1}} &   & 
*{}\ar@{-}[dl]^{l_{n}}\\
   & *{\bullet}\ar@{}|{\,\,\,\,\,\,\,\,\,\, v}\ar@{-}[d]^e\\
   & *{\bullet}\ar@{}|{\,\,\,\,\,\,\,\,\,\, w}\ar@{-}[d]\\
    & *{\bullet}\ar@{}|{\,\,\,\,\,\,\,\,\,\, x}\ar@{-}[d]\\
   & *{}\ar@{}|{\,\,\,\,\,\,\,\,\,\, .}}
\eeqn Now $\partial_z^*\partial_e^*$ will first send $q_e$ to $0$ and then 
$q_{l_1},\cdots, q_{l_n}$ to $0$. One the other hand $\partial_w^*\partial_v^*$ 
will first send $q_{l_1},\cdots , q_{l_n}$ to $0$ and then $q_e$ to $0$. The end 
result is evidently the same. 

For (VI) we begin with the commutative diagram 

\beqn
 \xymatrix{
 T \ar[r]^{\sigma_v} & T\backslash v \\
 T'\ar[u]^\partial\ar[r]^{\sigma_v} & T'\backslash v \ar[u]_\partial.
 }\eeqn Let us suppose that the face map $\partial$ removes edges $f_1,\cdots, 
f_n$. Since $T'$ still contains $v$ and its two adjacent edges (say $e_1$ and 
$e_2$), one can merge them to a new edge $e$. Thus $\partial^*$ is defined by 
$q_{f_i}\mapsto 0$ for $i=1,\cdots n$ and $\sigma_v^*$ by $q_e\mapsto q_{e_1} + 
q_{e_2}$. Hence it is clear that $\partial^*\sigma_v^* = \sigma_v^*\partial^*$. 
The verifications of (VII) and the special cases (see Remark \ref{specCase}) and 
similar and omitted.

Let us observe that $D(T)$ is unital for every $T\in\Omega$ and the 
$*$-homomorphisms $\partial_e^*,\partial_v^*,\sigma_v^*$ and $\theta^*$ are all 
unit preserving whence the essential image of the functor $D$ is indeed 
$\SCu^\op$. 

Note that for a map $\tau: S\map T$ in $\Omega$ the induced map is 
$\tau^*: D(T)\map D(S)$. It remains to check that the association $\tau \mapsto \tau^*$ 
respects composition of morphisms. It is clear that this association preserves 
composition of face maps as well as composition of degeneracy maps. To complete the 
proof we now simply invoke Remark \ref{uniqueFac}.
\end{proof}

\section{Draw-Dendraw adjunction and the Bridge} \label{DrawDendraw}
For a small category $\cC$ let $\P(\cC)$ denote the category of $\Set$-valued 
presheaves on $\cC$, i.e., $\Fun(\cC^\op,\Set)$. Thus setting $\cC=\Omega$ we 
find $\P(\Omega)=\dset$. Since $\P(\SCu^\op)$ is cocomplete, using the covariant 
functoriality of the category of presheaves (via left Kan extenstion) one obtains 
the dashed functor below:

\beq \label{draw}
\xymatrix{
\Omega\ar[r]^D \ar[d] & \SCu^\op \ar[d]\\
\dset \ar@{-->}[r] & \P(\SCu^\op),
}
\eeq where the vertical functors are the canonical Yoneda embeddings and the top 
horizontal functor $D:\Omega\functor\SCu^\op$ is the one constructed in the 
previous section (see Proposition \ref{functor}). Let $\draw$ denote the dashed 
functor in the above diagram \eqref{draw}. There is an adjunction 
$$\adj{\draw}{\dset}{\P(\SCu^\op)}{\den},$$ where the right adjoint $\den$ is 
defined as $[\den(Y)](T) = Y(D(T))$ for any $Y\in\P(\SCu^\op)$. 

\begin{defn} \label{drawDefn}
For any $X\in\dset$ the object $\draw(X)$ is its {\em $C^*$-algebraic drawing}. 
We call the functor $\draw$ (resp. $\den$) the {\em draw} (resp. {\em dendraw}) 
functor.
\end{defn}

\begin{rem} \label{colimits}
 In sheaf theoretic notation $\draw = D_!$ and $\den = D^*$. The dendraw functor 
$\den$ also admits a right adjoint $D_*:\dset\functor \P(\SCu^\op)$ whence it 
preserves colimits.
\end{rem}

\noindent
Recall from subsection \ref{denModel} that the category $\dset$ admits a 
combinatorial model structure. 

\begin{thm} \label{main}
There is a combinatorial model structure on $\P(\SCu^\op)$, such that the 
draw-dendraw adjunction $$\adj{\draw}{\dset}{\P(\SCu^\op)}{\den}$$ becomes a 
Quillen adjunction. 
\end{thm}

\begin{proof}
The model structure on $\P(\SCu^\op)$ that we are referring to is constructed in Theorem \ref{DenModel} 
(see the appendix in Section \ref{App}). The left adjoint $\draw$ 
sends generating cofibrations in $\dset$ to cofibrations in $\P(\SCu^\op)$ (see 
Proposition \ref{GenCof} below) and generating trivial cofibrations to trivial 
cofibrations in $\P(\SCu^\op)$ (see Remark \ref{TrivCof} below). Now using Lemma 
2.1.20 of \cite{Hov} one concludes that the draw-dendraw adjunction is actually 
a Quillen adjunction.
\end{proof}

\begin{rem} \label{infAdj}
 Associated with any (combinatorial) model category $\cM$ there is an underlying 
(presentable) $\infty$-category $\N(\cM^\circ)$ (see Definition 1.3.1 of 
\cite{HinLoc}). Moreover, a Quillen adjunction between (combinatorial) model 
categories [like $\adj{\draw}{\dset}{\P(\SCu^\op)}{\den}$] induces an 
$\infty$-categorical adjunction between the underlying (presentable) 
$\infty$-categories [like 
$\adj{\mathbf{L}\draw}{\N(\dset^\circ)}{\N(\P(\SCu^\op)^\circ)}{\mathbf{R}\den}$] 
(see Proposition 1.5.1 of \cite{HinLoc} and Theorem 2.1 of \cite{Maz-Gee}). 
Although we are mainly interested in the $\infty$-categorical adjunction pair 
$(\mathbf{L}\draw,\mathbf{R}\den)$, it is often convenient to have at our 
disposal an explicit Quillen adjunction modelling it.
\end{rem}

\begin{rem} \label{NewClassification}
Viewing $\SCu^\op$ inside the category of presheaves $\P(\SCu^\op)$ via the 
Yoneda functor we obtain a new homotopy theory for (the opposite category of) 
separable unital $C^*$-algebras, whose weak equivalences are called {\em weak 
operadic equivalences}. This new class of weak operadic equivalences is potentially interesting 
in its own right. The weak operadic equivalences on $\SCu^\op$ are different from 
those inherited from the model structure on $\Ind(\SCu^\op)$ (see \cite{BarJoaMe}) via 
the embedding $\SCu^\op\hookrightarrow\Ind(\SCu^\op)$. These two classes of weak equivalences 
on $\SCu^\op$ give rise to different homotopy theories. The class of weak operadic equivalences 
is not contained in the class of standard homotopy equivalences on $\SCu^\op$ (see Remark \ref{OperCont}); 
it is not clear to the author whether the other containment holds. Those readers who prefer to 
stick to the category of $C^*$-algebras (and not venture into the 
category of presheaves) may try to classify the objects in it up to weak operadic equivalences. 
\end{rem} 

\begin{rem}\label{opendSet}
A vertex that has no incoming edges is called a {\em stump}, e.g., in the 
$0$-corolla \beqn\xymatrix{*{\bullet}\ar@{-}[d]\\ {}\ar@{}|{\,\,\,\,\,\,\,\,\,\, 
}}\eeqn the top vertex is a stump. A tree devoid of stumps is called an {\em 
open tree}. Let $\Omega_o$ denote the full subcategory of $\Omega$ spanned by 
the open trees. The canonical inclusion $\Omega_o\hookrightarrow\Omega$ induces 
an adjunction $\dset_o:=\P(\Omega_o)\rightleftarrows\P(\Omega)=\dset$, such that 
the left adjoint $\dset_o\hookrightarrow\dset$ is fully faithful. The objects of 
$\dset_o$ are called {\em open dendroidal sets}. The category $\dset_o$ inherits 
a combinatorial model structure via the adjunction 
$\dset_o\rightleftarrows\dset$ making it a Quillen pair (see Section 2.3 of 
\cite{HeuHinMoe}). The fully faithful functor $\sset\functor\dset$ factors 
through $\dset_o$. The fibrant objects of $\dset_o$ are $\infty$-operads without 
constants. It was noticed by I. Moerdijk that our construction of the 
noncommutative dendrices functor does not distinguish between a leaf and an 
edge, whose top vertex is a stump; in particular, the $C^*$-algebra associated 
with the unit tree and the $0$-corolla are both $\CC$. Thus our draw-dendraw 
adjunction should be restricted to open dendroidal sets via the composite 
adjunction $$\dset_o\rightleftarrows \dset\rightleftarrows\P(\SCu^\op).$$ 
\end{rem}

So far we have constructed the solid adjunctions in the following diagram of 
$\infty$-categories: 
\beqn
\xymatrix{
\N(\dset_o^\circ)\ar@/^1pc/[rrd] &&&& 
\N(\P(\SCu^\op)^\circ)\ar@/^1pc/[lld]^{\mathbf{R}\den} \ar@/^1pc/@{-->}[rrd]\\
&& \N(\dset^\circ)\ar@/^1pc/[rru]^{\mathbf{L}\draw}\ar@/^1pc/[llu] &&&& \uNS . 
\ar@/^1pc/@{-->}[llu]
}
\eeqn Now we define the $\infty$-category of noncommutative spaces $\uNS$. Then 
we complete the connection between $\infty$-operads and noncommutative spaces 
via a sequence of $\infty$-categorical adjunctions. The dashed pair above 
actually represents a zigzag of adjunctions.

\subsection{The rest of the bridge between $\uNS$ and $\N(\P(\SCu^\op)^\circ)$} 
\label{otherAdj}
Earlier we constructed the compactly generated $\infty$-category of 
pointed noncommutative spaces generalizing the category of pointed compact noncommutative spaces (see Definition 2.13 of \cite{MyNSH}) . 
Let $\Csep^\op$ denote the opposite topological category of separable $C^*$-algebras with all 
(not necessarily unit preserving) $*$-homomorphisms. We view it as a topological category 
by endowing the morphism sets with the point-norm topology. Let $\iCsep^\op$ denote the topological 
nerve of $\Csep^\op$. It is shown in Proposition 2.7 of \cite{MyNSH} 
that $\iCsep^\op$ admits finite colimits.
\begin{defn}
We set $\iNS=\Ind_\omega(\iCsep^\op)$ and call it the compactly generated $\infty$-category of 
pointed noncommutative spaces.
\end{defn}
Similarly, there exists a compactly generated $\infty$-category $\uNS$ of noncommutative (unpointed) spaces 
whose construction is outlined below. 
\begin{defn}
Let $\cC$ denote the opposite of the topological category of separable unital 
$C^*$-algebras with unit preserving $*$-homomorphisms. We again view it as a topological category 
by endowing the morphism sets with the point-norm topology. 
\end{defn} Here we have included the zero $C^*$-algebra in 
the topological category $\cC$. The zero $C^*$-algebra should be viewed as the (unital) $C^*$-algebra of 
continuous functions on the empty space. Therefore, for every separable unital $C^*$-algebra $A$ 
there is a unique unital $*$-homomorphism $A\map 0$, i.e., the opposite category $\cC$ has an initial object. But the 
zero $*$-homomorphism $0\map A$ is {\em not unital} unless $A=0$.
\begin{defn}
Let $\NSf$ denote the topological nerve of the topological category $\cC$. Here it is vitally important to consider 
the point-norm topology on the morphism spaces while constructing the topological nerve.
\end{defn} One can show as in Proposition 2.7 
of \cite{MyNSH} that $\NSf$ admits finite colimits. For the rest of this section 
we set $\Ind=\Ind_\omega$ that denotes the $\infty$-categorical ind-completion.
\begin{defn} \label{NSpace}
We set $\uNS:=\Ind(\NSf)$ and call it the compactly generated $\infty$-category 
of (unpointed) noncommutative spaces.\end{defn} 

\begin{rem}
This $\infty$-categorical construction of noncommutative spaces $\uNS$ is 
simple and practical. It incorporates homotopy theory and analysis in a 
systematic manner; the analytical aspects are contained within the world of 
$C^*$-algebras. More complicated topological algebras like pro $C^*$-algebras 
can be viewed within this setup via the homotopy theory of diagrams of 
$C^*$-algebras. The mechanism is explained in our earlier work 
\cite{MyNSH,MyColoc}.
\end{rem}

There is a canonical fully faithful embedding of (topological) categories 
$\SCu^\op\hookrightarrow \cC$. This functor induces an adjunction of the 
corresponding categories of presheaves $\P(\SCu^\op)\rightleftarrows \P(\cC)$. A 
map $f: C\map D$ in $\cC$ is a {\em $C^*$-homotopy equivalence} if there is 
another map $g: D\map C$ and homotopies $fg\simeq\id_D$ and $gf\simeq\id_C$. The 
set of $C^*$-homotopy equivalences gives rise to a set of maps in $\P(\cC)$ that 
eventually gives rise to another set of maps in $\P(\SCu^\op)$ via the 
adjunction $\P(\SCu^\op)\rightleftarrows \P(\cC)$.

\begin{defn}[Mixed model structure on $\P(\SCu^\op)$] \label{mixModel}
 The left Bousfield localization of the combinatorial model category 
$\P(\SCu^\op)$ equipped with the operadic model structure (see Definition 
\ref{DenModel}) along the set of maps induced by the $C^*$-homotopy equivalences 
is the {\em mixed model structure} on $\P(\SCu^\op)$. We denote the mixed model 
category by $\P(\SCu^\op)_\mix$ that again turns out to be combinatorial.
\end{defn}

The Bousfield localization $\P(\SCu^\op)\map \P(\SCu^\op)_\mix$ of combinatorial 
model categories induces an adjunction of underlying presentable 
$\infty$-categories 
$\N(\P(\SCu^\op)^\circ)\rightleftarrows\N(\P(\SCu^\op)_\mix^\circ)$ that 
exhibits $\N(\P(\SCu^\op)_\mix^\circ)$ as a localization of 
$\N(\P(\SCu^\op)^\circ)$. Let $\theta$ denote the composition of the functors 
$$\cC\overset{j}{\hookrightarrow}\P(\cC)\functor\P(\SCu^\op)\overset{(-)^\f}{
\functor}\P(\SCu^\op)_\mix^\f ,$$ where $j$ is the Yoneda embedding, 
$\P(\SCu^\op)_\mix^\f$ is the full subcategory of (bi)fibrant objects of 
$\P(\SCu^\op)_\mix$, and $(-)^\f$ denotes a fibrant replacement functor in the 
mixed model category $\P(\SCu^\op)_\mix$. Let us view $\P(\SCu^\op)_\mix^\f$ as 
a {\em relative category} in the sense of \cite{BarKan} via the weak 
equivalences inherited from the model category $\P(\SCu^\op)_\mix$. We can also 
view $\cC$ as a relative category with the $C^*$-homotopy equivalences as the 
weak equivalences.

\begin{lem}
 The functor $\theta:\cC\map\P(\SCu^\op)_\mix^\f$ is a morphism of relative 
categories.
\end{lem}

\begin{proof}
We need to verify that $\theta$ preserves weak equivalences. Our construction of 
the mixed model category $\P(\SCu^\op)_\mix$ ensures this property (see Definition \ref{mixModel}).
\end{proof}

For any relative category $\cA$ we denote the underlying $\infty$-category by 
$\cA_\infty$ (see Section 1.2 of \cite{Maz-Gee}). The morphism of relative 
categories $\theta:\cC\map\P(\SCu^\op)_\mix^\f$ induces a morphism of underlying 
$\infty$-categories $\theta:\cC_\infty\map(\P(\SCu^\op)_\mix^\f)_\infty$. For any $\infty$-category $\cA$
there is an $\infty$-category of $\infty$-presheaves $\Pinf(\cA)$ (see \cite{LurToposBook}). Note the subtle difference 
in notation - for an ordinary category $\cA$ we denote by $\P(\cA)$ the category of $\Set$-valued presheaves on $\cA$; 
whereas for an $\infty$-category $\cA$ we denote by $\Pinf(\cA)$ the $\infty$-category of $\infty$-presheaves on $\cA$.

\begin{prop} \label{colimPres}
 The morphism of $\infty$-categories 
$\theta:\cC_\infty\map(\P(\SCu^\op)_\mix^\f)_\infty$ induces a colimit 
preserving functor 
$\tilde{\theta}:\Pinf(\cC_\infty)\map\N(\P(\SCu^\op)_\mix^\circ)$.
\end{prop}

\begin{proof}
The canonical inclusion $\P(\SCu^\op)_\mix^\f\hookrightarrow\P(\SCu^\op)_\mix$ 
induces an equivalence of underlying $\infty$-categories \cite{DwyKanHam} (see 
also Lemma 2.8 of \cite{Maz-Gee}). Thanks to the universal property of the 
category of presheaves $\Pinf(-)$ in the setting of $\infty$-categories (see 
Theorem 5.1.5.6 of \cite{LurToposBook}), it suffices to show that 
$(\P(\SCu^\op)_\mix^\f)_\infty\simeq\N(\P(\SCu^\op)_\mix^\circ)$ admits small 
colimits. Since the model category $\P(\SCu^\op)_\mix$ is combinatorial, its 
underlying $\infty$-category is presentable (see Corollary 1.5.2 of 
\cite{HinLoc}), i.e., it is cocomplete.
\end{proof}

\noindent
The following result is proven in Proposition 3.18 of \cite{BarJoaMe} using the 
formalism of weak (co)fibration categories \cite{BarSch}. 

\begin{lem} \label{presentable}
 There is an equivalence of $\infty$-categories $\Ind(\cC_\infty)\simeq\uNS$.
\end{lem}

\begin{rem}
Actually Proposition 3.18 of \cite{BarJoaMe} proves a pointed version of the 
above Lemma. The desired result can be shown using similar methods and hence its 
proof is omitted. 
\end{rem}

\begin{thm} \label{NSadj}
 There is a diagram of adjunctions of presentable $\infty$-categories: \beqn
\xymatrix{
&& \N(\P(\SCu^\op)_\mix^\circ)\ar@/_1pc/[rr]\ar@/^1pc/[lld] 
&&\Pinf(\cC_\infty)\ar@/_1pc/[ll]_{\tilde{\theta}} \ar@/^1pc/[rrd]\\
\N(\P(\SCu^\op)^\circ)\ar@/^1pc/[rru]^{} &&&&&& \Ind(\cC_\infty)\simeq\uNS . 
\ar@/^1pc/[llu]
}
\eeqn
\end{thm}

\begin{proof}
The presentability of each $\infty$-category in the above diagram is clear. 
Observe that $\tilde{\theta}:\Pinf(\cC_\infty)\functor\N(\P(\SCu^\op)_\mix^\circ)$ 
is a colimit preserving functor between presentable $\infty$-categories (see 
Proposition \ref{colimPres}). Hence using the Adjoint Functor Theorem (see 
Corollary 5.5.2.9 of \cite{LurToposBook}) we deduce that it admits a right 
adjoint. The existence of the adjunction pair $\Pinf(\cC_\infty)\rightleftarrows 
{\Ind(\cC_\infty)\simeq\uNS}$ is standard (see, for instance, Theorem 5.5.1.1 of 
\cite{LurToposBook}). The adjunction 
$\N(\P(\SCu^\op)^\circ)\rightleftarrows\N(\P(\SCu^\op)_\mix^\circ)$ has already 
been explained above.
\end{proof}

\begin{rem}
For the benefit of the reader we explain briefly the meaning and significance of 
this result. It is the author's perception that several results in the two paradigms 
of noncommutative geometry use very similar techniques, albeit in different contexts. 
For example, the constructions of the bivariant $\K$-theory category and the category of noncommutative 
motives are philosophically almost identical (only applied to different notions of spaces). 
That led to the vision of abstracting away the commonalities and providing a framework
whereby results can be tranferred back-and-forth creating synergies (cf. subsection 
\ref{BHVision}). In what follows we substantiate this assertion with a few potential 
directions of development.
\end{rem}

\section{Prospects: commutative spaces and graph algebras} \label{CNC}
It is known how to view commutatives spaces (or motives) inside their 
noncommutative counterparts in the algebro-geometric setting 
\cite{KonNotes,TabGarden,BluGepTab}. We briefly explain how the 
$\infty$-category of spaces (not necessarily compact) sits inside that of 
noncommutative spaces via a colocalization in the setting of Connes. We also 
highlight how noncommutative dendrices naturally interpolate between the two 
canonical notions of {\em building blocks}.

\subsection{Commutative spaces via colocalization} 
Let $\cS$ (resp. $\pS$) denote the $\infty$-category of spaces (resp. pointed 
spaces). It is shown in Theorem 1.9 (1) of \cite{MyColoc} that there is a fully 
faithful $\omega$-continuous functor $\pS\mono\iNS$. In the same vein one can 
show that there is a fully faithful $\omega$-continuous functor $\cS\mono\uNS$.

\begin{prop} \label{rAdjoint}
The fully faithful $\omega$-continuous functor $\pS\mono\iNS$ (as well as 
$\cS\mono\uNS$) admits a right adjoint, i.e., it is colimit preserving.
\end{prop}

\begin{proof}
 Due to the Gel'fand--Na{\u{\i}}mark correspondence there is a fully faithful 
functor $f:\cSf\mono\iCsep^\op$ that induces the fully faithful 
$\omega$-continuous functor $\Ind_\omega(f):\pS\mono\iNS$ of Theorem 1.9 (1) of 
\cite{MyColoc}. The functor $f$ preserves finite colimits whence it is right 
exact. Therefore, by Proposition 5.3.5.13 of \cite{LurHigAlg} the functor 
$\Ind_\omega(f)$ admits a right adjoint. The proof of the corresponding 
assertion for $\cS\mono\uNS$ is similar.
\end{proof}

\begin{defn}
We denote the right adjoint of $\pS\mono\iNS$ (resp. $\cS\mono\uNS$) in the 
above Proposition \ref{rAdjoint} by $\Com_*:\iNS\functor\pS$ (resp. 
$\Com:\uNS\functor\cS$) and call it the {\em underlying pointed space} (resp. 
{\em underlying space}) functor. Since $\Com_*$ and $\Com$ admit fully faithful 
left adjoints they are colocalizations, i.e., they constitute the commutative 
(pointed) space approximation of a noncommutative (pointed) space. 
\end{defn}

Now we are going to demonstrate how noncommutative dendrices interconnect 
simplices and matrices. Let $T_n$ denote the linear graph 
$$\bullet_0\overset{e_1}{\leftarrow}\bullet_1\overset{e_2}{\leftarrow}\cdots 
\overset{e_{n}}{\leftarrow}\bullet_{n},$$ whose graph algebra $C^*(T_n)$ is 
isomorphic to $M_{n+1}(\CC)$ (the construction of the graph algebra is explained 
below in subsection \ref{graphRes}). Let $D^\ab (T_n)$ denote the commutative 
unital $C^*$-algebra generated by requiring the generators 
$\{q_{e_1},\cdots,q_{e_n}\}$ of $D(T_n)$ to commute (see Remark \ref{abel}). 
There is a canonical surjective $*$-homomorphism $\pi_n:D(T_n)\map D^\ab (T_n)$ 
that is identity on the generators. It follows from Proposition 2.1 of 
\cite{CunSimp} that $D^\ab (T_n)$ is isomorphic to the commutative 
$C^*$-algebra $\C(\Delta^n)$. There is also a canonical $*$-homomorphism $s_n: 
D(T_n)\map C^*(T_{n-1})\cong M_{n}(\CC)$, sending $q_{e_i}\mapsto e_{ii}$. Note that 
$\sum_{i=1}^n e_{ii}$ is the identity matrix that is the unit in the graph 
algebra $C^*(T_{n-1})\cong M_n(\CC)$. Thus we have a zigzag of arrows 
 \beq \label{SM}
 \xymatrix{
 && D(T_n)\ar[dll]_{\pi_n}\ar[drr]^{s_n} \\
 D^\ab (T_n)\cong\C(\Delta^n) &&&& C^*(T_{n-1})\cong M_n(\CC).
 }
 \eeq The set of $*$-homomorphisms $\{s_n \,|\, n\in\NN\}$ defines a set of maps 
$M$ in the $\infty$-category noncommutative spaces $\uNS$ via the functor 
$j:\NSf\map\uNS$. Thus we are going to invert the maps in $M$ to construct the 
simplex-matrix identified version of $\uNS$. It is quite natural to consider 
matrix algebras as noncommutative simplices.

\begin{defn}
 The accessible localization $L_M:\uNS\map M^{-1}\uNS=:\iNSm$ that admits a fully faithful 
 right adjoint is defined to be the $\infty$-category of {\em simplex-matrix identified noncommutative spaces}. 
\end{defn}

\begin{rem}
Since $\uNS$ is a presentable $\infty$-category, so is $\iNSm$.  
\end{rem}

\begin{rem} \label{tract}
The composite functor $\iNSm\mono\uNS\overset{\Com}{\functor}\cS$ defines the 
underlying space functor on $\iNSm$. The subcategory of simplex-matrix 
identified noncommutative spaces $\iNSm$ is a tractable part of the entire 
$\infty$-category of noncommutative spaces $\uNS$ and it would be nice to 
explore it further.
\end{rem}

\begin{rem} \label{ncgreal}
Let $\CW$ denote the category of finite CW complexes. The geometric realization functor 
$|\cdot |:\sset\map \Ind(\CW)$ preserves (tensor) products
and detects weak equivalences, whose counterpart in the world of dendroidal sets has been treated 
in \cite{Heuts,BasNik}. It is plausible (and desirable) that one can modify the 
functor $\draw:\dset\functor\P(\SCu^\op)$ to produce yet another {\em 
$C^*$-algebraic} or {\em noncommutative} geometric realization of dendroidal 
sets that fits into the following commutative diagram:

\beqn
\xymatrix{
\sset \ar[rrr]^{|\cdot|}\ar[d] &&& \Ind(\CW)\ar[d]^{}\\
\dset \ar[rrr]^{?} &&& \Ind(\SCu^\op)\subset\P(\SCu^\op).
}
\eeqn
We leave it as an open problem. 
\end{rem}

\subsection{Graph algebras} \label{graphRes}
There is a vast literature on graph algebras (or graph $C^*$-algebras) with 
several interesting results relating structural aspects of the graph algebra 
(like simplicity) to purely graph theoretic properties. We encourage the 
interested readers to consult, for instance, \cite{RaeBook} and the references 
therein.

Let $E$ be a finite {\em directed} graph and let $\cH$ be a fixed separable Hilbert space. A 
{\em Cuntz--Krieger $E$-family $\{S,P\}$ on $\cH$} (abbreviated as CK 
$E$-family) consists of a set $P=\{P_v\,|\, v\in E^0\}$ of mutually orthogonal 
projections on $\cH$ and a set $S=\{S_e\,|\, e\in E^1\}$ of partial isometries 
on $\cH$, such that \begin{enumerate}                                                                              
\item (CK1) $S_e^* S_e = P_{s(e)}$ for all $e\in E^1$; and 
\item (CK2) $P_v = 
\sum_{\{e\in E^1\,:\, r(e)=v\}} S_e S_e^*$ provided $\{e\in E^1\,:\, r(e)=v\}\neq\emptyset$.
\end{enumerate} The graph 
algebra of $E$, denoted by $C^*(E)$, is by definition the universal 
$C^*$-algebra generated by $\{S,P\}$ subject to relations (CK1) and (CK2). It is 
known that $C^*(E)$ is unital if and only if the set of vertices $E^0$ is finite 
(see Proposition 1.4 of \cite{KumPasRae}).
                                                                             
\begin{rem}
Some authors prefer to write the relations (CK1) and (CK2) differently, viz., 
the roles of $r$ and $s$ are interchanged. We have adopted the convention from 
\cite{RaeBook}. The advantage of this viewpoint is that juxtaposition of edges 
in a path corresponds to composition of partial isometries on the Hilbert space 
$\cH$.\end{rem}                                                                           

\begin{ex}
 The graph algebra corresponding to the graph 
 \begin{tikzcd}
\arrow[loop left]{l}{} \bullet \arrow[loop right]{r}{}
\end{tikzcd} is Cuntz algebra $\cO_2$.
\end{ex}

The left Quillen functor $\draw:\dset\functor\P(\SCu^\op)$ is obtained by the 
left Kan extension of $\Omega\overset{D}{\map}\SCu^\op\map \P(\SCu^\op)$ along 
$\Omega\functor\dset$. Explicitly it is given by the formula: $$[\draw (X)] (A) 
= \underset{f: D(T)\map A}{\colim} X(T),$$ where the colimit is taken over the 
comma category $(D\downarrow A)$. The Quillen adjunction descends to an 
adjunction of homotopy categories 
$$\adj{\mathbf{L}\draw}{\Ho(\dset)}{\Ho(\P(\SCu^\op))}{\mathbf{R}\den},$$ after 
taking the total derived functors of $\draw$ and $\den$ ($\mathbf{L}\draw$ and 
$\mathbf{R}\den$ respectively).

The composite $\mathbf{L}\draw\circ\mathbf{R}\den$ defines a comonad on 
$\Ho(\P(\SCu^\op))$. Viewing any separable unital $C^*$-algebra $A$ inside 
$\Ho(\P(\SCu^\op))$ via the Yoneda functor, we may consider the map given by the 
counit of the adunction $\mathbf{L}\draw\circ\mathbf{R}\den(A)\functor\Id(A)$. 
It is presumably not an isomorphism; nevertheless, one should consider its 
comonadic resolution. If $A$ is a graph algebra, this resolution can be viewed 
as a {\em resolution of the underlying graph by trees}. It would be nice to 
classify $C^*$-algebras up to this dendroidal invariant.

\begin{rem} \label{KirchbergAlg}
In the world of $C^*$-algebras a celebrated result of Kirchberg asserts that 
topological $\K$-theory acts as a complete invariant on the subcategory of so-called
{\em stable Kirchberg algebras} that satisfy UCT \cite{KirPhi}. It was 
shown in \cite{MyComparison,CorPhi} that for such $C^*$-algebras (in fact, for a larger subcategory 
of $C^*$-algebras) algebraic $\K$-theory is naturally isomorphic to 
topological $\K$-theory (see Theorem 2.4 and Remark 1 of \cite{MyComparison}). If the vision outlined in 
the introduction can be realised, viz., if one can show that algebraic $\K$-theory and $\KK$-theory 
can be recovered from diagram \eqref{vision}, then the above-mentioned construction would provide a 
{\em higher invariant} that has the potential to act as a complete invariant on a bigger subcategory than
that of stable Kirchberg algebras satisfying UCT. Observe that topological $\K$-theory is also the primary 
classification tool for graph algebras. It would be actually more prudent to analyse this construction 
for a graph algebra at the level of underlying $\infty$-categories (and not at 
the level of homotopy categories), possibly, after passing to the stabilization.
\end{rem}

\section{Appendix: The model structure on $\P(\SCu^\op)$} \label{App}
For any small category $\cC$ there is a {\em Cisinski model structure} on 
$\P(\cC)$ \cite{CisModel}, whose construction is described below. A {\em 
functorial cylinder object} is an endofunctor 
$I\otimes(-):\P(\cC)\functor\P(\cC)$, such that for every $X\in\P(\cC)$ there 
are natural morphisms $\partial^0_X,\partial^1_X,\sigma_X$ that satisfy:
\begin{enumerate}
 \item the following diagram commutes: \beqn
\xymatrix{
X\ar[rd]_{\partial^0_X}\ar@/^1pc/[rrrd]^{\id_X}\\
& I\otimes X \ar[rr]^{\sigma_X} && X ,\\
X \ar[ru]^{\partial^1_X}\ar@/_1pc/[rrru]_{\id_X}
}
\eeqn 

\item the canonical morphism $X\coprod X\map I\otimes X$ induced by 
$\partial^0_X,\partial^1_X$ is a monomorphism.
\end{enumerate} The choice of a functorial cylinder object 
$\cJ=(I\otimes(-),\partial^0_{(-)},\partial^1_{(-)},\sigma_{(-)})$ constitutes 
an {\em elementary homotopical datum} if $\cJ$ satisfies the following two 
additional conditions:

\begin{enumerate}[(i)]
 \item the functor $I\otimes(-)$ commutes with small colimits, and
 \item for every monomorphism $j:K\map L$ in $\P(\cC)$ for $e=0,1$ the diagram 
\beqn
 \xymatrix{
 K \ar[r]^j\ar[d]_{\partial^e_K} & L\ar[d]^{\partial^e_L}\\
 I\otimes K \ar[r]^{I\otimes j} & I\otimes L
 }
 \eeqn is a pullback square.
\end{enumerate} Using the functorial cylinder object $\cJ$ on can define an {\em 
elementary $\cJ$-homotopy} between two maps in $\P(\cC)$, viz., two maps $f,g: 
X\map Y$ are elementary $\cJ$-homotopic if there is a map $\eta: I\otimes X \map 
Y$ making the following diagram commute:

\beqn
\xymatrix{
X\ar[rd]^f\ar[d]_{\partial^0_X}\\
I\otimes X \ar[r]^\eta & Y\\
X .\ar[u]^{\partial^1_X}\ar[ru]_g}
\eeqn Let $\Ho_{\cJ}\P(\cC)$ denote the category whose objects are those of 
$\P(\cC)$ and whose morphisms are the elementary $\cJ$-homotopy classes of 
morphisms of $\P(\cC)$. 

\begin{defn} \label{Jhtpy}
There is a canonical functor $\P(\cC)\functor\Ho_{\cJ}\P(\cC)$ and the 
morphisms that descend to isomorphisms under this functor are called {\em 
$\cJ$-homotopy equivalences}. This notion obviously depends on the choice of $\cJ$. \end{defn} 
 
\noindent
 The model structure on $\P(\cC)$ depends on another choice, viz., a class $\An$ 
of {\em anodyne extensions}. For a class $M$ of maps of $\P(\cC)$ we denote by 
$\llp(M)$ [resp. $\rlp(M)$] the class of maps that satisfy left [resp. right] 
lifting property with respect to $M$. For any cartesian square \beqn
\xymatrix{
X\ar[r]\ar[d] & Y\ar[d]\\
Z \ar[r] & W}\eeqn in $\P(\cC)$ with $Y\map W$ and $Z\map W$ monomorphisms, the 
canonical map $Y\coprod_X Z\map W$ is also a monomorphism. For brevity this 
monomorphism is suggestively written as $Y\cup Z \map W$. 

\begin{defn}
 Let $\cJ$ be an elementary homotopy datum on $\P(\cC)$. Then a {\em class of 
anodyne extensions $\An$ relative to $\cJ$} is a class of morphisms in $\P(\cC)$, 
such that 
 \begin{enumerate}[(a)]
  \item $\An = \llp(\rlp(M))$ for a small set of maps $M$,
  \item for any monomorphism $K\map L$ and for $e=0,1$ the induced map $I\otimes 
K \cup \{e\}\otimes L \map I\otimes L$ belongs to $\An$, and
  \item if $K\map L$ belongs to $\An$, then so does $I\otimes K\cup \partial 
I\otimes L \map I\otimes L$, where $\partial I \otimes L = L\coprod L$.
 \end{enumerate}
\end{defn} 

\begin{rem} \label{AnEx}
It is shown in Proposition 1.3.13 of \cite{CisModel} that for any small set $S$ 
of monomorphisms of $\P(\cC)$ there is a smallest class of anodyne extensions 
relative to $\cJ$ that is generated by $S$. This class of morphisms is denoted 
by $\An_{\cJ}(S)$.
\end{rem}

\begin{thm}[Th{\'e}or{\`e}me 1.3.22 of \cite{CisModel}] \label{CisMS}
Let $\cJ$ be an elementary homotopy datum on $\P(\cC)$ and $\An_{\cJ}(S)$ be a 
class of anodyne extensions relative to $\cJ$ that is generated by a small set 
$S$ of monomorphisms. Then there is a combinatorial model structure on $\P(\cC)$ 
satisfying
\begin{enumerate}
 \item the cofibrations are the monomorphisms,
 \item $X\in\P(\cC)$ is fibrant if the map $X\map \star$, where $\star$ is the terminal 
object, satisfies right lifting property with respect to all anodyne extensions 
$\An_{\cJ}(S)$, and 
 \item a map $f:X\map Y$ is a weak equivalence if for all fibrant objects $Z$ 
the induced map $f^*:\Ho_{\cJ}\P(\cC)(Y,Z)\map \Ho_{\cJ}\P(\cC)(X,Z)$ is 
bijective.
\end{enumerate}
\end{thm}

\begin{rem}
 The Cisinksi model structure on $\P(\cC)$ admits a functorial fibrant 
replacement. A set of generating cofibrations can be chosen to be those 
monomorphisms whose codomains are quotients of representable presheaves (see 
Proposition 1.2.27 of \cite{CisModel}). Every object of $\P(\cC)$ is cofibrant 
and its homotopy category is equivalent to the full subcategory of 
$\Ho_{\cJ}\P(\cC)$ spanned by the fibrant objects (see 1.3.23 of 
\cite{CisModel}). Moreover, a morphism between two fibrant objects is a weak 
equivalence if and only if it is a $\cJ$-homotopy equivalence.
\end{rem}

\begin{prop} \label{GenCof}
 The functor $\draw:\dset\functor\P(\SCu^\op)$ preseves cofibrations.
\end{prop}

\begin{proof}
The set of generating cofibrations in $\dset$ is $\{\partial \Omega[T]\map 
\Omega[T]\,|\, T\in\Omega \}$. Each face map $\partial: T'\map T$ of trees 
induces a monomorphism of representable presheaves, whose image is specified by 
the datum of this monomorphism of representable presheaves (see Chapter IV of 
\cite{MacMoe}). For any tree $T$ the boundary inclusion $\partial \Omega[T] 
\map\Omega[T]$ is obtained as a union of the images of such face maps. We know 
that $\draw$ sends the representable presheaf of $T$ to that of $D(T)$. Each 
face map $\partial:T'\map T$ in $\Omega$ induces a surjective $*$-homomorphism 
$\partial^*:D(T)\map D(T')$ in $\SCu$ (see subsection \ref{func}). It induces a 
monomorphism in $\SCu^\op$ and the Yoneda embedding preserves monomorphisms 
whence $\draw(\partial):\SCu^\op(-,D(T'))\map \SCu^\op(-,D(T))$ is a 
monomorphism in $\P(\SCu^\op)$. It follows from the universal property of the 
noncommutative dendrices construction that $\draw$ sends the generating 
cofibrations of $\dset$ to monomorphisms of $\P(\SCu^\op)$. Note that the 
cofibrations of $\P(\SCu^\op)$ are precisely the monomorphisms whence Lemma 
2.1.20 of \cite{Hov} shows that $\draw$ preserves cofibrations.
\end{proof}

\begin{rem}
 It is clear that the above Proposition does not depend on the choice of $\cJ$.
\end{rem}

\noindent
For the choice of the elementary homotopy datum we have a few possibilities at 
our disposal. 

\begin{ex}[Example 1.3.9 of \cite{CisModel}] \label{Lawvere}
Let $\cC$ be any small category. For an object $C\in\cC$ let us denote the 
representable presheaf of $C$ in $\P(\cC)$ by $h_C$. Let $\cL$ denote the 
presheaf that associates with every $C\in\cC$ the set $\cL(C)=\{\text{subobjects 
of $h_C$}\}$. For every map $u:C\map D$ in $\cC$ the map $\cL(D)\map\cL(C)$ is 
induced by pullback along $u$. The presheaf $\cL$ turns out to be a subobject 
classifier, i.e., $\P(\cC)(X,\cL)\simeq\{\text{subobjects of the presheaf 
$X$}\}$. If $\star$ is the final object of $\P(\cC)$, then it has exactly two 
subobjects $\star\hookrightarrow \star$ and $\emptyset\hookrightarrow \star$, 
where $\emptyset$ denotes the initial object of $\P(\cC)$. These define two 
morphisms $\lambda_0,\lambda_1: \star\map \cL$. The tuple 
$(\cL,\lambda_0,\lambda_1)$ gives rise to an elementary homotopy datum by 
setting $I\otimes X = \cL\times X$, $\partial^e_X = \lambda_e\times\id_X$, 
$e=0,1$, and $\sigma_X = \mathrm{pr}_2:\cL\times X\map X$. This elementary 
homotopy datum is called the {\em Lawvere cylinder} that exists in any category 
of presheaves like $\P(\SCu^\op)$.
\end{ex}

\begin{ex}For any {\em nonzero} separable unital $C^*$-algebra $A$ there is a 
sequence of two $*$-homomorphisms $A\overset{\iota}{\map} 
A[0,1]:=\C([0,1],A)\overset{\ev_t}{\map} A$ for any $t\in [0,1]$ (natural in $A$), whose composition 
is the identity $*$-homomorphism on $A$. Here $\iota(a)$ is the constant 
$a$-valued function on $[0,1]$ for every $a\in A$ and $\ev_t$ is the evaluation 
at $t\in [0,1]$. For $A=\CC$ after reversing the arrows and passing to the 
representable presheaves in $\P(\SCu^\op)$ we get the following square \beq 
\label{cyl}
\xymatrix{
\emptyset \ar[rr] \ar[d] && h_\CC \ar[d]^{\partial^1=\ev_1^*}\\
h_\CC \ar[rr]^{\partial^0=\ev_0^*} && h_{\C([0,1]),}
}
\eeq where $\emptyset$ is the initial object (empty presheaf) of $\P(\SCu^\op)$. Note 
that $\P(\SCu^\op)$ are $\Set$-valued covariant functors on $\SCu$ and we do not notationally 
distinguish between objects in a category and in its opposite.
For every $A\in\SCu^\op$ we find that the following diagram

\beqn
\xymatrix{
\emptyset \ar[r] \ar[d] & h_\CC(A) \ar[d]^{\ev_1^*}\\
h_\CC (A) \ar[r]^{\ev_0^*} & h_{\C([0,1])}(A)
}
\eeqn is a pullback square in $\Set$. Indeed, $ h_\CC(A) =\SCu^\op(A,\CC) 
=\{\one_A\}$, where $\one_A$ is the unique unital $*$-homomorphism $\CC\map A$, 
and $(\one_A \circ\ev_t^*) (f) = f(t)\one_A$ for $t=0,1$ and for every 
$f\in\CC[0,1]=\C([0,1],\CC)$. In this argument it is crucial that $A$ is a {\em nonzero} 
separable unital $C^*$-algebra. Since limits are 
computed objectwise in $\P(\SCu^\op)$ we conclude that diagram \eqref{cyl} is a 
pullback square. It follows from Example 1.3.8 of \cite{CisModel} that 
$$\cJ=(I\times X,\partial^0\times \id_X,\partial^1\times\id_X,\mathrm{pr}_X: 
I\times X\map X)$$ defines an elementary homotopy datum. 
\end{ex}

\begin{ex}[Continuous cylinder] \label{ContCyl}
Consider again the sequence of $*$-homomorphisms $A\overset{\iota}{\map} 
A[0,1]\overset{\ev_t}{\map} A$ (natural in $A$), whose composition is the 
identity $*$-homomorphism on $A$. Given any representable object $h_A$ we set 
$I\otimes h_A = h_{A[0,1]}$ and extend the cylinder construction to all objects 
of $\P(\SCu^\op)$ by commuting with colimits, i.e., if $X\cong\colim_i\, 
h_{A_i}$, then we set $I\otimes X \cong \colim_i\, h_{A_i[0,1]}$.
\end{ex}

We choose the elementary homotopy datum of Example \ref{Lawvere} since it is the 
most canonical choice for the Cisinski model structure on any presheaf category. Subsequently
we are going to localize our model structure based on our requirements.
Let $X$ be a set of generating trivial cofibrations of $\dset$ and set $S=\draw(X)$. By the 
above Proposition \ref{GenCof} $S$ is a set of monomorphisms of $\P(\SCu^\op)$ 
that generates a class of anodyne extensions $\An_{\cJ}(S)$ relative to 
$\cJ$ (see Remark \ref{AnEx}). As a consequence of Theorem \ref{CisMS} we obtain

\begin{thm}[Operadic model structure] \label{DenModel}
 With the choice of the elementary homotopy datum $\cJ$ of Example \ref{Lawvere} 
and the class of anodyne extensions $\An_{\cJ}(S)$ relative to $\cJ$ described 
above $\P(\SCu^\op)$ acquires the structure of a combinatorial model category.
\end{thm}

\begin{rem} \label{OperCont}
 Note that the Lawvere cylinder is different from the continuous cylinder of Example 
\ref{ContCyl}. Hence the evaluation map $A[0,1]\overset{\ev_t}{\map} A$ is not a weak 
equivalence in the operadic model structure; it roughly mirrors the Joyal model structure on the 
category of simplicial sets, in which $\Delta^1\map\Delta^0$ is not a weak equivalence.
\end{rem}

\begin{rem} \label{TrivCof}
It is shown in Lemma 1.3.31 of \cite{CisModel} 
that every anodyne extension is a 
weak equivalence. Since $\draw(X)=S\subset \An_{\cJ}(S)$, where $X$ is the set 
of generating trivial cofibrations of $\dset$, we observe that by construction 
the functor $\draw$ sends generating trivial cofibrations of $\dset$ to trivial 
cofibrations of $\P(\SCu^\op)$.
\end{rem}

\begin{rem} \label{GenCisModel}
The construction of the Cisinski model structure can be profitably used in other 
contexts. For instance, one can start with a small category $\cA$ of topological 
algebras (Banach, Fr{\'e}chet, or locally convex) with some mild hypotheses. 
Then one can simply start with the minimal model structure on $\P(\cA^\op)$ by 
choosing the Lawvere cylinder (see Example \ref{Lawvere}) for the elementary 
homotopy datum $\cJ$  and $\An_\cJ(\emptyset)$ for the class of anodyne 
extensions. Now one can localize this combinatorial model category by inverting 
a small set of morphisms like differentiable homotopy equivalences between the 
representable objects in $\P(\cA^\op)$. This would produce an unstable model 
category to start with that can be ($\infty$-categorically) stabilized and 
localized further according to one's requirements; for instance, one can aim for 
a stable $\infty$-category, whose morphism groups model the Cuntz $\kk$-groups 
for locally convex algebras \cite{CunBivSurvey}. {\O}stv{\ae}r developed his homotopy 
theory of $C^*$-algebras adopting a similar strategy in the setting of cubical set valued 
presheaves on the category of separable $C^*$-algebras \cite{Ost} but we do not 
expect a Quillen equivalence between his unstable model category for {\em 
cubical $C^*$-spaces} and $\P(\SCu^\op)$ equipped with the operadic model 
structure as in Theorem \ref{DenModel}. This is because the evaluation map $A[0,1]\overset{\ev_t}{\map} A$ of 
the continuous cylinder construction (see Example \ref{ContCyl}) is not a weak equivalence in 
the operadic model structure. One final observation - all the ingredients needed
to develop a Waldhausen $\K$-theory of noncommutative spaces are now at our 
disposal.
\end{rem}

%----------------------------------------bibliography---------------------------

\bibliographystyle{abbrv}

\bibliography{/home/ibatu/Professional/math/MasterBib/bibliography}

\vspace{5mm}

\end{document}